\documentclass[reqno]{amsart}

\usepackage{hyperref}
\usepackage{amsmath}
\usepackage{amssymb}
\usepackage{amsfonts}
\usepackage{amsthm}
\usepackage{latexsym}
\usepackage{esint}
\usepackage{mathtools}
\usepackage{mathrsfs}
\usepackage{graphicx}
\usepackage{subcaption}
\usepackage{wrapfig}
\usepackage{bbm}
\usepackage{color}
\usepackage{bm}
\usepackage{float}

\usepackage{tikz}
\usepackage{tikz-3dplot}

\usepackage{tikz-cd}
\usetikzlibrary{lindenmayersystems}
\usetikzlibrary{decorations.pathreplacing}

\usepackage[top=1in, bottom=1.25in, left=1.10in, right=1.10in]{geometry}
\usepackage{todonotes}
\usepackage{stackengine}
\allowdisplaybreaks

\newcommand{\test}{{\mathscr F^\eta}}

\newcommand{\weakto}{\rightharpoonup}

\newtheorem{theorem}{Theorem}[section]
\newtheorem{lemma}[theorem]{Lemma}
\newtheorem{proposition}[theorem]{Proposition}
\newtheorem{corollary}[theorem]{Corollary}

\theoremstyle{definition}
\newtheorem{definition}[theorem]{Definition}

\theoremstyle{remark}
\newtheorem{remark}[theorem]{Remark}

\numberwithin{equation}{section}

\newcommand{\bx}{\mathbf{x}}
\newcommand{\by}{{y}}

\newcommand{\bu}{\mathbf{u}}
\newcommand{\bfu}{\mathbf{u}}
\newcommand{\bfn}{\mathbf{n}}

\newcommand{\bfw}{\mathbf{w}}
\newcommand{\bff}{\mathbf{f}}

\newcommand{\bv}{\mathbf{v}}
\newcommand{\bn}{\mathbf{n}}

\newcommand{\bfpsi}{\boldsymbol{\psi}}
\newcommand{\bfPsi}{\boldsymbol{\Psi}}
\newcommand{\bfphi}{\boldsymbol{\phi}}
\newcommand{\bfvarphi}{\boldsymbol{\varphi}}

\newcommand{\dy}{\, \mathrm{d}y }
\newcommand{\dd}{\,\mathrm{d}}

\newcommand{\dx}{\, \mathrm{d} \mathbf{x}}

\newcommand{\dt}{\, \mathrm{d}t}
\newcommand{\ds}{\, \mathrm{d}s}

\newcommand{\dxt}{\,\mathrm{d}\mathbf{x}\, \mathrm{d}t}

\newcommand{\Div}{\mathrm{div}_{\mathbf{x}}}
\newcommand{\divx}{\mathrm{div}_{\mathbf{x}}}

\newcommand{\nabx}{\nabla_{\mathbf{x}}}
\newcommand{\naby}{\partial_{y}}

\newcommand{\delx}{\Delta_{\mathbf{x}}}

\newcommand{\R}{\mathbb{R}}

\newcommand{\Oeta}{\mathcal{O}_{\eta}}

  \providecommand{\Xint}[1]{\mathchoice
    {\XXint\displaystyle\textstyle{#1}}%
    {\XXint\textstyle\scriptstyle{#1}}%
    {\XXint\scriptstyle\scriptscriptstyle{#1}}%
    {\XXint\scriptscriptstyle\scriptscriptstyle{#1}}%
    \!\int}
  \providecommand{\XXint}[3]{{\setbox0=\hbox{$#1{#2#3}{\int}$}
      \vcenter{\hbox{$#2#3$}}\kern-.5\wd0}}

\begin{document}

\title{Strong transport noise prevents self-intersection in 2D fluid-structure interaction.}




\author{Dominic Breit}
\address{Faculty of Mathematics, University of Duisburg-Essen, Thea-Leymann-Strasse 9, 45127 Essen,
Germany}
\email{dominic.breit@uni-due.de}

\author{Prince Romeo Mensah}
\address{African Institute for Mathematical Sciences, 1 Shoppers Street, Accra, Ghana}
\email{romeo@aims.edu.gh}
\subjclass[2010]{35R60; 60H15; 74B20; 74E35}

\date{\today}

\begin{abstract}

We consider the interaction of a viscous incompressible fluid with a flexible elastic shell occupying the boundary of the fluid's domain in two space dimensions. The shell is modelled by a visco-elastic linearised beam equation. It is an open problem in fluid-structure interaction, if a self-intersection of the structure can be excluded in the long run.
We show that by adding a physically well-motived
and sufficiently strong transport noise in the structure equation, the self-intersection does not occur before any given end-time with any prescribed probability smaller than 1.  
\end{abstract}

\keywords{Koiter shell, transport noise, no-contact, fluid-structure interaction }

\maketitle

\section{Introduction}
We consider the interaction of a viscous incompressible fluid in two dimensions with a flexible structure that is located at the boundary of the fluid's domain. This leads to the fluid's domain moving in time in accordance with the evolution of the structure. The fluid is described by the Navier--Stokes equations, while the structure displacement
satisfies a (visco-)elastic beam equation
\begin{align}\label{eq:1}
\varrho_s\partial_t^2\eta-\gamma\partial_t\partial_y^2\eta+\alpha\partial_y^4\eta=g-( \mathbb{S}\circ \bm{\varphi}_\eta\bn_\eta )\cdot\bn \,|\naby\bm{\varphi}_\eta|.
\end{align}
The fully coupled system is given in \eqref{detNSF} below following an introduction of the necessary mathematical framework. 
In \eqref{eq:1}, $\mathbb S$ is the stress tensor of the fluid, $\bfn$ is the outward unit normal of the reference domain $\mathcal O$, $\bfn_\eta$ is the outward unit normal of the moving domain $\mathcal O_\eta$, and $g$ is a given force. The function $\bfvarphi_\eta$ gives a coordinate transform between the reference and moving domain, see Figure \ref{fig:2}. The parameters $\varrho_s,\gamma,\alpha>0$ will be put to 1 for simplicity of the presentation. 
We also follow the usual convention of neglecting the linear second order differential operator that complements the leading term  $\alpha\partial_y^4\eta$.
As in \cite{BMSS} we distinguish two cases of elastic structures at the boundary visualised in Figure \ref{fig:1}:
\begin{itemize}
\item \textbf{Elastic plates.} The reference geometry for the fluid domain is flat and the elastic plate is located on top of a rectangle. In this context, the transformation map between the reference geometry and the moving domain is particularly simple.
\item \textbf{Elastic shells.} In the case of a general reference geometry, where the structure occupies
(a part of) the curved boundary, we speak of an 
elastic shell. Here, the transformation map is significantly more complicated than for plates.
\end{itemize}
\begin{figure} 
\begin{center}
\begin{tikzpicture}[scale=2]
  \begin{scope}  
     \draw [thick, <->] (0.5,0.5) -- (0.5,-0.5) -- (2.8,-0.5);
        \node at (0.5,0.6) {$\R$};
       \node at (3.1,-0.5) {$\R$};
       \draw [blue,dashed] (0.5,0.2) -- (2.5,0.2);
                \draw (2.5,0.2) -- (2.5,-0.5);
                \draw [red] (0.5,0.2) .. controls (0.7,-0.1) and (1,0.8) .. (1.3,0.1);
                \draw [red] (1.3,0.1) .. controls (1.5,-0.1) and  (1.7,0.4) .. (2.5,0.2);
                        \node at (1.4,-0.25) {\textcolor{blue}{$\mathcal{O}$} $\stackrel{\bfvarphi_\eta}{\rightsquigarrow}$\textcolor{red}{$\mathcal{O}_\eta$}};
                         \node [red] at (1.5,0.4) {$\eta$};
        \draw [thick, <->] (4.5,0.5) -- (4.5,-0.5) -- (6.8,-0.5);
        \node at (4.5,0.6) {$\R$};
       \node at (7.1,-0.5) {$\R$};
                     \draw (4.5,-0.5) .. controls (4,0) and (4.5,0.2) .. (4.5,0.2);
       \draw (6.5,0) .. controls (6.8,-0.2) and (6.5,-0.9) .. (4.5,-0.5);
        \draw [blue,dashed] (4.5,0.2) .. controls (4.7,0.3) and (4.7,0.3) .. (4.7,0.3);
        \draw [blue,dashed] (4.7,0.3) .. controls (4.8,0.4) and (4.9,0.3) .. (5.1,0.2);
       \draw [blue,dashed] (5.1,0.2) .. controls (5.3,-0.3) and (5.5,0.3) .. (6.5,0);
                \draw [red] (4.5,0.2) .. controls (4.7,-0.1) and (5,0.8) .. (5.3,0.1);
                \draw [red] (5.3,0.1) .. controls (5.5,-0.1) and  (5.7,0.4) .. (6.5,0);
                        \node at (5.4,-0.25) {\textcolor{blue}{$\mathcal{O}$} $\stackrel{\bfvarphi_\eta}{\rightsquigarrow}$ \textcolor{red}{$\mathcal{O}_\eta$}};
                         \node [red] at (5.5,0.4) {$\eta$};
  \end{scope}
\end{tikzpicture}
\caption{Domain transformation: elastic plates (left) and shells (right), cf.~\cite[fig. 1]{BMSS}}\label{fig:1}
\end{center}
\end{figure}
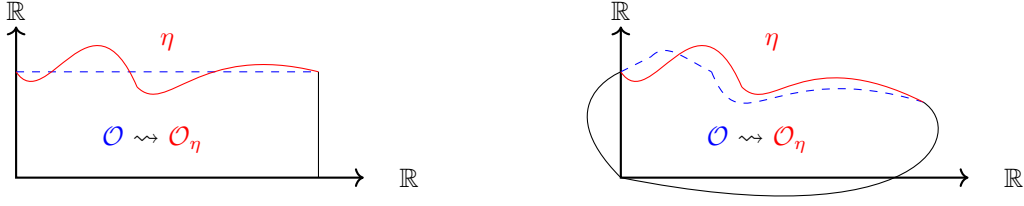
Recently, there has been  remarkable progress on the well-posedness theory for the fluid-structure problem governed by \eqref{eq:1}.
The first result regarding visco-elastic plates can be found in \cite{MR3466847}.
Eventually, this result was extended by the first author in \cite{breit2022regularity} to visco-elastic beams (i.e., a one-dimensional visco-elastic shell).
 Recently, even the case of perfectly elastic plates (where $\gamma=0$ in \eqref{eq:1}) based on a new acceleration estimate was established in \cite{ScSu} and extended to shells
in our paper \cite{BMSS}.
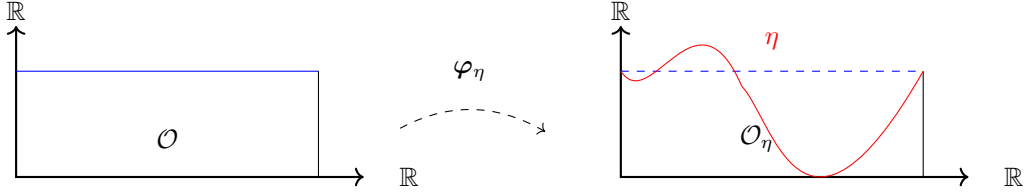
\begin{figure} 
\begin{center}
\begin{tikzpicture}[scale=2]
  \begin{scope}     \draw [thick, <->] (0.5,0.5) -- (0.5,-0.5) -- (2.8,-0.5);
        \node at (0.5,0.6) {$\R$};
        \node at (1.5,-0.25) {\textcolor{blue}{$\mathcal{O}$}};
       \node at (3.1,-0.5) {$\R$};
        \draw [blue] (0.5,0.2) -- (2.5,0.2);
                \draw (2.5,0.2) -- (2.5,-0.5);
                   \node at (3.5,0.2) {$\bfvarphi_\eta$};
        \draw [thick, <->] (4.5,0.5) -- (4.5,-0.5) -- (6.8,-0.5);
        \node at (4.5,0.6) {$\R$};
       \node at (7.1,-0.5) {$\R$};
       \draw [blue,dashed] (4.5,0.2) -- (6.5,0.2);
                \draw (6.5,0.2) -- (6.5,-0.5);
                \draw [red] (4.5,0.2) .. controls (4.7,-0.1) and (5,0.8) .. (5.3,0.1);
                \draw [red] (5.3,0.1) .. controls (5.5,-0.1) and  (5.7,-1.2) .. (6.5,0.2);
                        \node at (5.1,-0.25) {\textcolor{red}{$\mathcal{O}_\eta$}};
                         \node [red] at (5.5,0.4) {$\eta$};
          \draw [<-,dashed] (4,-0.2) to [out=150,in=30] (3,-0.2);
  \end{scope}
\end{tikzpicture}
\caption{Interaction with contact, cf.~\cite[fig. 2]{BrRo}.}\label{fig:2}
\end{center}
\end{figure} 

In addition to well-posedness, a second fundamental question that is of great importance in fluid-structure interaction problems is whether or not the structure self-intersects, see figure \ref{fig:2}. In many real-world applications, for instance the flow of  blood through vessels, such an intersection is excluded by natural mechanisms. It is, however, very difficult to rigorously verify this mathematically. The only available result is given in \cite{MR3466847} for damped plates (see also the recent paper \cite{Liu} for perfectly elastic plates). It is proved that the elastic wall cannot touch the bottom of the fluid cavity in finite time. 
The approach heavily depends on the simplified geometry in the plate-setting based on a particular stream function.  
It remains unclear if a similar theory is possible for elastic shells.
In this paper we are going to shed new light on this question by way of probability. 

\subsection{Stochastic fluid-structure interaction} 
Stochastic perturbations in PDEs do well to account for random effects in real-life problems and uncertainty in a data. In the context of fluid-structure interaction problems, incorporating stochasticity into structure
models can be particularly useful for analysing stochastic buckling \cite{papadopoulos2009buckling, reuter2025probabilistic, schafer2006stochastic} caused by geometric  imperfections,   material variability, random loads and many more. With that said, it is also important how said stochasticity is incorporated in fluid-structure interaction models. Having  highlighted the importance of stochasticity in fluid-structure interaction problems in \cite{KuCa1}, a first step towards a well-posedness theory for such an interaction problem is presented in \cite{KuCa2}. Here, the fluid is modelled by the two-dimensional Stokes equations while the structure is the  one-dimensional wave equation. This simplified model has since been extended to the full Navier--Stokes equation by the same authors and their collaborators in  \cite{KuCa1, KuCa2, KuTa1, Ta1, KuCa3}. A different direction is considered in \cite{BrMeMo, MeNg} where stochasticity is not taken  as an external force but as an intrinsic property of the system by way of a transport noise.  Transport noise  has the very appealing feature of being energy conservative. Thus, if the initial data is deterministic (or simply bounded in probability) we have a pathwise control over the energy. Transport noise also has a clear physical meaning and can be derived rigorously from a stochastic variational principle  \cite{Holm3} as shown in \cite{MeNg} by the second author and his collaborator. Additionally, it has been observed to have regularising effects on certain ill-posed partial differential equations \cite{flandoli2021scaling, flandoli2020convergence,Fl3, galeati2020convergence}.  With these benefits in mind, the well-posedness theory for stochastic fluid-structure interaction with transport noise has been initiated in  \cite{BrMeMo} where martingale solutions for the fluid-structure problem with structure equation roughly given by
\begin{align}\label{eq:sfsi}\dd \partial_t \eta- \partial_t\naby^2\eta\dt + \naby^4 \eta\dt
=\big(g-( \mathbb{S}\circ \bm{\varphi}_\eta\bn_\eta )\cdot\bn \,|\naby\bm{\varphi}_\eta|\big)\dt+
\sqrt{\gamma}
\sum_{k\in\mathbb{Z}} ( {\sigma}_k \naby\partial_t\eta)\circ\dd W_t^k
\end{align}
 are constructed. 
The precise understanding of the equation  is detailed in Section \ref{sec:strat} below. In this paper we continue studying the effect of such a noise on the dynamics.
In  \cite{BrMeMo} we studied the interaction of a three-dimensional fluid with a two-dimensional structure, where the $\sigma_k$ above are solenoidal vector fields. Thus the noise is energy conservative. This assumption has no suitable analogue in the 2D/1D interaction we are considering here. Assuming instead that the $\sigma_k$ are spatially constant trivialises the situation: Starting with a finite sum one obtains linear combination of Wiener process leads to an object which coincides in law with a single Wiener process.
We will argue in Theorem \ref{thm:mainA} that the approach from \cite{BrMeMo} is robust under removing the solenoidability assumption. Note also that there is no physical motivation for solenoidal noise
coefficients (except from being energy conservative). Considering a truly (asymptotically) infinite dimensional noise in \eqref{eq:sfsi} will lead to an interesting effect on the evolution of the system.

\subsection{The main result}
Our main result shows that by adding a physically well-motived and sufficiently strong transport noise in the structure equation, the self-intersection in the stochastic fluid-structure interaction problem does not occur before any given end-time with any prescribed probability smaller than 1.  
In a simplified form, it is given as follows (see Theorem \ref{thm:main} for the precise formulation).
\begin{theorem}
Let $\kappa\in(0,1)$, and $T>0$ be given. Then there are $\gamma>0$ and a sequence $(\sigma_k)_{k\in\mathbb Z}$ with finite support such that for any weak martingale solution $((\Omega,\mathfrak{F},(\mathfrak{F}_t)_{t\geq0},\mathbb{P}),\eta,\bu,(W^k)_{k\in\mathbb Z})$ of the fluid-structure interaction problem with structure equation given by \eqref{eq:sfsi}
there is no contact until the time $T$ with probability $1-\kappa$.
\end{theorem}
The construction of the coefficients $(\sigma_k)_{k\in\mathbb Z}$ is motivated from the analysis in \cite{Fl3}, where a similar noise sequence creates strong dissipation in the scaling limit. As in  \cite{Fl3}, the effect in the deterministic target problem is rather obvious. The difficulty is to transfer some information to the primitive problem, for which no additional estimate becomes available by the noise (as being energy conservative). This is a regularisation by noise effect to which we give some context in the following.

\subsection{Regularisation by noise: contextualisation}
The possibility to regularise simple finite dimensional stochastic evolution equations has been known for a long time. While certain linear PDEs could be considered earlier, a major step was the regularisation of the transport equation in \cite{Fl0}. Finally, in \cite{Fl3}, it has been shown that a carefully chosen transport noise can delay the vorticity blow-up in three-dimensional Navier--Stokes equations (see also \cite{agresti1}  for further progress in this regard).
In \cite{Fl3}  transport noise is added to the vorticity equation, while transport noise in the momentum equation seems to be more adequate. This gap was closed very recently in
\cite{agresti2} where a physically well-motived noise is considered.

The construction of our noise is borrowed from \cite{Fl3}.
Our problem is much simpler since the regularisation effect only concerns the linear one-dimensional structure equation. However, we must deal with a noise which is not energy conservative.
This requires some fine-tuning of the coefficients, see \eqref{ansatz}. We also benefit from one-dimensional embeddings for the structure displacement used in a sharp form, see \eqref{new}.
The scaling limit for the uncoupled structure equation  (see \cite{MeNg}) is rather straightforward. The main difficulty is the procedure for the full problem with nonlinear coupling and a fluid domain which is moving in time.
To some extent, we benefit from our recent  works on the deterministic target problem:
\begin{itemize}
\item In \cite{breit2022regularity} the first author showed the existence of a unique global-in-time strong solution which exists up to the point of contact.
\item In \cite{BMSS35} we showed a weak-strong uniqueness principle. The latter is conditioned on the Lipschitz continuity of the structure function which is unconditionally true in the two-dimensional case.
\end{itemize}
Both points  are crucial for the Gy\H{o}ngy-Krylov argument in Section \ref{sec:cip} to conclude on uniqueness of weak solutions with regular initial data. Interestingly, the uniqueness of weak solutions to the deterministic target system with initial data in the energy space is not known and probably not even expected. Comparing two distinct solutions is a sophisticated task as both velocity fields are a priori defined on different moving domains.

We also benefit from our previous analysis in \cite{BrMeMo}, where martingale solutions are constructed. However, several arguments from \cite{BrMeMo} must be strengthened. In particular, we improve the fractional estimate for the structure function
(which has its origin in \cite{MuSc}), see \eqref{eq:13.05B}. Although we only improve the differentiability by an instant (in the Besov scale, it only concerns the fine index), this improvement is crucial to use the solution as a test-function. Different from the two-dimensional Navier--Stokes equations in a fixed domain, this step is highly non-trivial because of the moving boundary and its limited regularity. Most of the analysis from Section \ref{sec:pcon}
 is dedicated to this step.

\section{Deterministic framework}
\subsection{Conventions}
For simplicity, we set all physical constants (except for $\gamma>0$) to 1. The analysis is not affected as long as they are strictly positive.
For two non-negative quantities $f$ and $g$, we write $f\lesssim g$  if there is a constant $c>0$ such that $f\leq\,c g$. Here $c$ is a generic constant which does not depend on the crucial quantities. If necessary, we specify particular dependencies. We write $f\approx g$ if both $f\lesssim g$ and $g\lesssim f$ hold. In the notation for function spaces (see next subsection), we do not distinguish between scalar- and vector-valued functions. However, vector-valued functions will usually be denoted in bold case. We shorten the time interval $(0,T)$ by $I$. We identify $\Gamma$ with the one-dimensional torus and thus consider the structure equation supplemented with periodic boundary conditions.

\subsection{Classical function spaces}
Let $\mathcal O\subset\R^2$ be open.
The function spaces of continuous or $\alpha$-H\"older-continuous functions, $\alpha\in(0,1)$,
 are denoted by $C(\overline{\mathcal O})$ or $C^{0,\alpha}(\overline{\mathcal O})$ respectively, where $\overline{\mathcal O}$ is the closure of $\mathcal O$. Similarly, we write $C^1(\overline{\mathcal O})$ and $C^{1,\alpha}(\overline{\mathcal O})$ if the functions are differentiable.
We denote  by $L^p(\mathcal O)$ and $W^{k,p}(\mathcal O)$ for $p\in[1,\infty]$ and $k\in\mathbb N$, the usual Lebesgue and Sobolev spaces over $\mathcal O$. 
 We denote by $W^{k,p}_0(\mathcal O)$, the closure of the smooth and compactly supported functions in $W^{k,p}(\mathcal O)$. If $\partial\mathcal O$ is regular enough, this coincides with the functions vanishing $\mathcal H^1$ -a.e. on $\partial\mathcal O$ with the one-dimensional Hausdorff measure $\mathcal H^1$. 
 We also denote by $W^{-k,p'}(\mathcal O)$ the dual of $W^{k,p}_0(\mathcal O)$ for $p,p'\in[1,\infty)$, with $\frac{1}{p}+\frac{1}{p'}=1$.
  Finally, we consider subspaces
$W^{1,p}_{\Div}(\mathcal O)$ and $W^{1,p}_{0,\Div}(\mathcal O)$ of divergence-free vector fields which are defined accordingly. The space $L^p_{\Div}(\mathcal O)$ is defined as the closure of the set of smooth and compactly supported solenoidal functions in $L^p(\mathcal O).$ 

For $p\in[1,\infty)$, the fractional Sobolev space (Sobolev-Slobodeckij space) with differentiability $s>0$ with $s\notin\mathbb N$ will be denoted by $W^{s,p}(\mathcal O)$. 
 We denote by $W^{s,p}_0(\mathcal O)$, the closure of the smooth and compactly supported functions in $W^{s,p}(\mathcal O)$. For $s>\frac{1}{p}$ this coincides with the functions vanishing $\mathcal H^{1}$ -a.e. on $\partial\mathcal O$ provided that $\partial\mathcal O$ is regular enough. We also denote by $W^{-s,p'}(\mathcal O)$, for $s>0$ and $p,p'\in[1,\infty)$, with $\frac{1}{p}+\frac{1}{p'}=1$, the dual of $W^{s,p}_0(\mathcal O)$.

Fractional derivatives on $\Gamma$ can be characterised with the help 
of the discrete Fourier coefficients
\begin{align*}
c_{ k}(f)=\int_\Gamma f(y)
e^{-\mathfrak i k y}\dd y,\quad k\in\mathbb Z.
\end{align*}
One can then define
\begin{align*}
H^{s,p}(\Gamma)&:=\big\{f\in \mathscr S'(\Gamma):\,\,\|f\|_{H^{s,p}(\Gamma)}<\infty\big\},\\
\|f\|_{H^{s,p}(\Gamma)}^p&:=\bigg\|\sum_{ k\in\mathbb Z }(1+| k|^2)^{s/2}c_{k}(f)e^{-\mathfrak i k \cdot }\bigg\|_{L^p(\Gamma)}.
\end{align*}
In the above $\mathscr S'(\Gamma)$ denotes the dual
of the Schwarz class $\mathscr S(\Gamma)$. For $p=2$ the space $H^{s,p}(\Gamma)$ coincides with the corresponding Sobolev-Slobodeckij space $W^{s,p}(\Gamma)$ defined via double integrals.

For a separable Banach space $(X,\|\cdot\|_X)$, we denote by $L^p(I;X)$, the set of (Bochner-) measurable functions $u:I\rightarrow X$ such that the mapping $t\mapsto \|u(t)\|_{X}$ belongs to $L^p(I)$. 
The set $C(\overline{I};X)$ denotes the space of functions $u:\overline{I}\rightarrow X$ which are continuous with respect to the norm topology on $(X,\|\cdot\|_X)$, while $C_w(\overline{I};X)$ denotes continuity with respect to the weak topology. Similarly, $C_{w^*}(\overline{I};X)$ denotes continuity with respect to the weak* topology.
 For $\alpha\in(0,1]$ we write
$C^{0,\alpha}(\overline{I};X)$ for the space of H\"older-continuous functions with values in $X$. The space $W^{1,p}(I;X)$ consists of those functions from $L^p(I;X)$ for which the distributional time derivative belongs to $L^p(I;X)$ as well. The space $W^{k,p}(I;X)$ is defined accordingly.

\subsection{Geometric setup}
\label{sec:geometry}
Following \cite{breitSchw2018compressible,lengeler2014weak,MuSc}, we present the geometric set-up for moving domains.
 The spatial reference domain $\mathcal O$ is assumed to be an open bounded subset of $\mathbb{R}^2$ with smooth boundary $\partial\mathcal O$ and an outer unit normal ${\bn}$. We also assume that
 $\partial\mathcal O$ can be parametrised by an injective mapping $\bm{\varphi}\in C^k(\Gamma;\R^2)$ for some sufficiently large $k\in\mathbb N$, where $\Gamma$ is the one-dimensional torus. Next, we suppose for all $\by\in \Gamma$ that  $\partial_\by\bm{\varphi}\neq\bm{0}$. Now,
 for a point $\bx$ in the neighbourhood
of $\partial\mathcal O$, we define the functions $\by$ and $s$ by  
\begin{align*}
 \by(\bx)=\arg\min_{\by\in\Gamma}|\bx-\bm{\varphi}(\by)|,\quad s(\bx)=(\bx-\mathbf p(\bx))\cdot\bn(\by(\bx)),
 \end{align*}
where we used the projection $\mathbf p(\bx)=\bm{\varphi}(\by(\bx))$. We define $L>0$ to be the largest number such that $s,\by$ and $\mathbf p$ are well-defined on $S_L$, where
\begin{align}\label{eq:boundary1}
S_L=\{\bx\in\R^2:\,\mathrm{dist}(\bx,\partial\mathcal O)<L\}
\end{align}
is a tubular neighbourhood of the boundary.\footnote{Note carefully, that this is not necessarily a small number. For instance, if $\mathcal O$ is a ball then $L$ is its radius.} Note that
due to the smoothness of $\partial\mathcal O$, for $L$ small enough, we have $|s(\bx)|=\min_{\by\in\Gamma}|\bx-\bm{\varphi}(\by)|$ for all $\bx\in S_L$. 
Now, for a given function $\eta : I \times \Gamma \rightarrow\R$ we parametrise the deformed boundary by
\begin{align}\label{eq:bfvarphi}
{\bm{\varphi}}_\eta(t,\by)={\bm{\varphi}}(\by) + \eta(t,\by){\bn}(\by), \quad \,\by \in \Gamma,\,t\in I.
\end{align}
With an abuse of notation we define the deformed space-time cylinder as $I\times\mathcal O_\eta=\bigcup_{t\in I}\{t\}\times\mathcal O_{\eta(t)}\subset\R^{3}$.
The corresponding function spaces for variable domains are defined as follows.
\begin{definition}{(Function spaces)}
For $I=(0,T)$, $T>0$, and $\eta\in C(\overline{I}\times\Gamma)$ with $\|\eta\|_{L^\infty(I\times\Gamma)}< L$ we define for $1\leq p,r\leq\infty$
\begin{align*}
L^p(I;L^r(\mathcal O_\eta))&:=\big\{v\in L^1(I\times\mathcal O_\eta):\,\,v(t,\cdot)\in L^r(\mathcal O_{\eta(t)})\,\,\text{for a.e. }t,\,\,\|v(t,\cdot)\|_{L^r(\mathcal O_{\eta(t)})}\in L^p(I)\big\},\\
L^p(I;W^{1,r}(\mathcal O_\eta))&:=\big\{v\in L^p(I;L^r(\mathcal O_\eta)):\,\,\nabx v\in L^p(I;L^r(\mathcal O_\eta))\big\}.
\end{align*}
\end{definition}
\noindent 
To establish a relationship between the 
fixed domain and the time-dependent domain, we introduce the Hanzawa transform $\bm{\Psi}_\eta : \mathcal O \rightarrow\mathcal O_\eta$ defined by
\begin{equation}
\label{map}
\bm{\Psi}_\eta(\bx)
=
 \left\{
  \begin{array}{lr}
    \mathbf{p}(\bx)+\big(s(\bx)+\eta(\by(\bx))\phi(s(\bx))\big)\bn(\by(\bx)) &\text{if dist}(\bx,\partial\mathcal O)<L,\\
    \bx &\text{elsewhere}
  \end{array}
\right.
\end{equation}
for any $\eta:\Gamma\rightarrow (-L,L)$. Here $\phi\in C^\infty(\mathbb R)$ is such that 
$\phi\equiv 0$ in a neighborhood of $-L$ and $\phi\equiv 1$ in a neighborhood of $0$. The other variables  $\mathbf{p}$, $s$ and $\bn$ are as defined earlier in this section.
A straightforward verification shows that the inverse of $\bm{\Psi}_{\eta(t)}$ is $\bm{\Psi}_{-\eta(t)}$.

We finish this section by recalling  the following Aubin--Lions-type lemma which is shown in~\cite[Theorem 5.1. \& Remark 5.2.]{MuSc} and slightly reformulated for our purposes.
\begin{theorem}
\label{thm:auba}
Let $X,Z$ be two Banach spaces, such that $X'\subset Z'$.
Assume that $f_N:I\to X$ and $g_N: I\to X'$, such that $g_N\in L^\infty(I;Z')$ uniformly. Moreover assume the following: 
\begin{enumerate}
\item[(a)] The {\em boundedness}: For some $\mathfrak s\in [1,\infty]$ we have that $(f_N)$ is bounded in $L^{\mathfrak s}(I;X)$ and $(g_N)$ in $L^{\mathfrak s'}(I;X')$.
\item[(b)] The {\em approximability-condition} is satisfied: For every $\kappa\in (0,1]$ there exists some $f_{N,\kappa}\in L^{\mathfrak s}(I;X)\cap L^1(I;Z)$,
%
such that for every $\epsilon\in (0,1)$ there exists some $\kappa_\epsilon\in(0,1)$ (depending only on $\epsilon$) such that 
\[
\|f_N-f_{N,\kappa}\|_{L^{\mathfrak s}(I;X)}\leq \epsilon\text{ for all } \kappa\in (0,\kappa_\epsilon]
\]
and for every $\kappa\in (0,1]$ there is some $C(\kappa)$ such that
\[
\|f_{N,\kappa}\|_{L^1(I;Z)}\dt\leq C(\kappa).
\]

\item[(c)] The {\em equi-continuity} of $g_N$: We require that there exists some $\alpha\in (0,1]$,  functions $A_N$ with $A_N\in L^1(I)$ uniformly, such that for every $\kappa>0$ that there exist some $C(\kappa)>0$ and some $N_\kappa\in \mathbb N$ such that for $\tau>0$, $N\geq N_\kappa$ and a.e.\ $t\in [0,T-\tau]$
\[
 \Big|\tau^{-1}\int_{0}^\tau\langle f_{N,\kappa}(t), g_N(t)-g_N(t+s)\rangle_{X,X'}\,\dd s\Big| \leq C(\kappa)\tau^\alpha(A_N(t)+1).
\]

\item[(d)] The {\em compactness assumption} is satisfied: $X'\hookrightarrow \hookrightarrow Z'$. More precisely, every uniformly bounded sequence in $X'$ has a strongly converging subsequence in $Z'$. 
\end{enumerate}
Then there is a subsequence, such that
\[
\int_I\langle f_N,g_N\rangle_{X,X'}\dt\rightarrow \int_I\langle f,g\rangle_{X,X'}\dt.
\]
\end{theorem}

\subsection{The deterministic problem}
In the scaling limit we obtain the following deterministic fluid-structure interaction system
\begin{equation}
\begin{aligned}\label{detNSF}  
&\divx \overline{\bu} =0 & \text{in } I\times\mathcal{O}_{\overline{\eta}}, 
\\
&\partial_t \overline{\bu}  + (\overline{\bu}\cdot \nabx)\overline{\bu}= 
 \delx \overline{\bu} -\nabx \overline{p} +\bff& \text{in } I\times\mathcal{O}_{\overline{\eta}},  
\\
&\partial_t^2 \overline{\eta} -(1+\gamma)\naby^2\partial_t\overline{\eta}  
+ \naby^4 \overline{\eta}
=
g-( \mathbb{S}\circ \bm{\varphi}_\eta\bn_\eta )\cdot\bn \,|\naby\bm{\varphi}_\eta|
& \text{in } I\times\Gamma,
\\
&\overline{\bu}\circ\bm{\varphi}_{\overline{\eta}}=(\partial_t\overline{\eta} )\bn & \text{on }I\times \Gamma,
\\
&\overline{\eta}(0,\cdot)=\eta_0(\cdot)  & \text{on } \Gamma,
\\
&\partial_t\overline{\eta}(0,\cdot)=\eta_\star(\cdot) & \text{on } \Gamma,
\\
&\overline{\bu}(0,\cdot)=\bu_0(\cdot) & \text{in }\mathcal{O}_{\eta_0},
\end{aligned}
\end{equation}
with $\gamma>0.$
Here we have a purely deterministic notion of a weak solution for the system of PDE \eqref{detNSF} given as follows:
\begin{definition}[Weak solution] \label{def:weakSolution}
Let $(\bff, g, \eta_0, \eta_\star,  \bu_0)$ be a dataset such that
\begin{equation}
\begin{aligned}
\label{dataset}
&\bff\in L^2(I;L^2_{\mathrm{loc}}(\mathbb{R}^2)), \quad g\in L^2(I;L^2(\Gamma)),
\\&
\eta_0 \in W^{2,2}(\Gamma ) \text{ with } \Vert \eta_0 \Vert_{L^\infty( \Gamma )} < L, \quad
\eta_\star \in L^{2}(\Gamma ),  
\\
&\bu_0\in L^{2}_{\mathrm{\divx}}(\mathcal{O}_{\eta_0} ) \text{ is such that }
 \bu_0 \circ\bm{\varphi}_{\eta_0}   =\eta_\star \bn \text{ on }  \Gamma.
\end{aligned}
\end{equation} 
We call the tuple
$(\overline{\eta},\overline{\bu})$
a weak solution of \eqref{detNSF} with data $(\bff, g, \eta_0,  \eta_\star,\bu_0)$ provided that the following holds:
\begin{itemize}
\item[(a)] The structure displacement $\overline{\eta}$ satisfies
\begin{align*}
\overline{\eta} \in L^\infty \big(I; W^{2,2}(\Gamma) \big)  \cap W^{1,\infty} \big(I; L^2(\Gamma) \big)\cap W^{1,2} \big(I; W^{1,2}(\Gamma) \big)
\quad \text{with} \quad \Vert \eta \Vert_{L^\infty(I \times \Gamma)} <L,
\end{align*}
as well as $\overline{\eta}(0)=\eta_0$ and $\partial_t\overline{\eta}(0)=\eta_\star$.
\item[(b)] The velocity field $\overline{\bu}$ satisfies $\overline{\bu} \circ\bm{\varphi}_{\overline{\eta}}   =\partial_t \overline{\eta}  \bn \text{ on } I\times \Gamma$ and
\begin{align*}
\overline{\bu} \in L^2 \big(I; W^{1,2}(\mathcal{O}_{\overline{\eta}}) \big)\cap L^\infty \big(I; L^2_{\Div}(\mathcal{O}_{\overline{\eta}}) \big) 
\end{align*}
as well as $\overline{\bu}(0)=\bfu_0$.
\item[(c)] For all  $(\phi, {\bfphi}) \in C^\infty(\overline{I}\times\Gamma) \times C^\infty_{\Div}(\overline{I}\times \R^2; \R^2)$ with $\phi(T,\cdot)=0$, ${\bfphi}(T,\cdot)=0$ and $\bfphi\circ\bfvarphi_\eta =\phi {\bfn}$ on $I\times\Gamma$, we have
\begin{align*}
\int_I  \frac{\mathrm{d}}{\dt}\bigg(
\int_\Gamma \partial_t \overline{\eta} \, \phi \dy
+&\int_{\mathcal{O}_{\overline{\eta}}}\overline{\bu}  \cdot {\bfphi}\dx
\bigg)\dt 
\\=&\int_I \int_\Gamma \big(\partial_t \overline{\eta}\, \partial_t\phi - \partial_y^2\overline{\eta}\,\partial_y^2 \phi
-(1+\gamma)\partial_t\partial_y\overline{\eta}\,\partial_y \phi+
 g\, \phi \big)\dy\dt 
\\&+
\int_I  \int_{\mathcal{O}_{\overline{\eta}}}\big( \overline{\bu}\cdot \partial_t  {\bfphi} + \overline{\bu} \otimes \overline{\bu}: \nabx {\bfphi} 
-
\nabx \overline{\bu}:\nabx {\bfphi} +\bff\cdot{\bfphi}
  \big) \dx\dt .
\end{align*} 
\item[(c)] For a.a. $t\in I$, we have
\begin{equation}
\begin{aligned}
\label{energyEst}
 \Vert\partial_t\overline{\eta} (t)\Vert_{L^2(\Gamma)}^2&+ \Vert \naby^2\overline{\eta}(t)\Vert_{L^2(\Gamma)}^2
+\Vert\overline{\bu}(t)\Vert_{L^2(\mathcal{O}_{\overline{\eta}})}^2 
+ (1+\gamma)\int_0^t \Vert \partial_y\partial_t\overline{\eta} \Vert_{L^2(\Gamma)}^2 \ds
+ \int_0^t \Vert \nabx\overline{\bu} \Vert_{L^2(\mathcal{O}_{\overline{\eta}})}^2 \ds
\\
\lesssim&
 \Vert \eta_\star \Vert_{L^2(\Gamma)}^2+  \Vert \naby^2\eta_0 \Vert_{L^2(\Gamma)}^2
+\Vert\bu_0\Vert_{L^2(\mathcal{O}_{\eta_0})}^2
\\&+
 \int_0^t \Vert g \Vert_{L^2(\Gamma)}^2 \ds
 +
  \int_0^t \Vert \bff\Vert_{L^2(\mathcal{O}_{\overline{\eta}})}^2 \ds.
\end{aligned}
\end{equation} 
\end{itemize}
\end{definition}
The existence of a weak solution can be shown as in \cite{lengeler2014weak} (also see \cite{muha2013existence} for the cylindrical Koiter model). The term $\gamma\partial_t\partial_y^2\eta$ is not included there, but it does not alter the arguments as it actually regularises the system. 
Note that we use a pressure-free formulation (that is, with test-function satisfying additionally $\Div\bfphi=0$) here. If the solution possesses more regularity,
the pressure can be recovered. This is the case for a strong solution which is defined as follows.

\begin{definition}[Strong solution] \label{def:strongSolution}
We call the triple $(\overline{\eta},\overline{\bu},\overline{p})$ a strong solution of \eqref{detNSF} provided $(\overline{\eta},\overline{\bu})$ is a weak solution of \eqref{detNSF}, it satisfies
\begin{align*}
\overline{\eta} \in W^{1,\infty} \big(I; W^{1,2}(\Gamma) \big)\cap W^{1,2} \big(I; W^{2,2}(\Gamma) \big)\cap  L^\infty \big(I; W^{3,2}(\Gamma) \big) \cap W^{2,2}(I;L^2(\Gamma)),\\
\overline{\bu} \in W^{1,2} \big(I; L^{2}(\mathcal{O}_{\overline{\eta}}) \big)\cap  L^2 \big(I; W^{2,2}(\mathcal{O}_{\overline{\eta}}) \big),
\quad
\overline{p}\in  L^2 \big(I; W^{1,2}(\mathcal{O}_{\overline{\eta}}) \big),
\end{align*}
and equations \eqref{detNSF}$_1$--\eqref{detNSF}$_3$ are satisfied a.a. in space-time.
\end{definition}
The following result is proved in \cite{breit2022regularity}.
\begin{theorem}\label{thm:Br}
Suppose that the dataset $(\bff, g, \eta_0, \eta_\star,  \bfu_0)$
satisfies in addition to \eqref{dataset} that
\begin{equation}
\begin{aligned}
\label{dataset'}
%
\eta_0 \in W^{3,2}(\Gamma) , \quad \eta_\star \in W^{1,2}(\Gamma),\quad 
\bfu_0\in W^{1,2}(\mathcal{O}_{\eta_0}).
\end{aligned}
\end{equation}
Then there is a unique strong solution $(\overline{\eta},\overline{\bu},\overline{p})$  of \eqref{detNSF} in the sense of Definition \ref{def:strongSolution}. The interval of existence is of the form $I = (0, t)$, where $t < T$ only in case $\mathcal{O}_{\overline{\eta}(s)}$ approaches a self-intersection when $s\rightarrow t$ or it degenerates
(namely, if $\lim_{s\rightarrow t}\partial_y\bfvarphi_\eta(s,\Gamma)=0$ or $\lim_{s\rightarrow t}\bfn(y)\cdot\bfn_{\eta(s)}(y)=0$ for some $y\in\Gamma$).
\end{theorem}

\subsection{Extension in moving domains}
\label{sec:ext}

We finally present a linear solenoidal extension. It is borrowed from \cite{MuSc}.
To this end, we first consider the \textit{corrector map}
\begin{align}\label{eq:corr}
\mathscr K_\eta:L^1(\Gamma )\rightarrow\mathbb{R},\qquad \mathscr K_\eta(\xi)=\frac{\int_{\mathcal{A}_\kappa}{\xi}(\by(\bx))\lambda_\eta(t,\bx)\dx}{\int_{\mathcal{A}_\kappa}\lambda_\eta(t,\bx)\dx},
\end{align}
where $\lambda_\eta\geq 0$  for $(t,\bx)\in I\times \mathcal{A}_\kappa$ is an appropriately chosen weight function, cf. \cite[equ. (3.3)]{MuSc}, and $\mathcal{A}_\kappa:=S_{L-\kappa/2}\setminus S_{L-\kappa}$ for some small $\kappa>0$. It satisfies
\begin{align}\nonumber
&\Vert\mathscr K_\eta(\xi)\Vert_{L^q(I)}
\lesssim \Vert\xi\Vert_{L^q(I;L^1(\Gamma ))}
\\
&\Vert\partial_t\mathscr K_\eta(\xi)\Vert_{L^q(I)}
\lesssim \Vert\partial_t\xi\Vert_{L^q(I;L^1(\Gamma ))}
+
\Vert\xi\partial_t\eta\Vert_{L^q(I;L^1(\Gamma ))}\label{eq:corr}
\end{align}
for all $q\in[1,\infty]$. The corrector $\mathscr K_\eta$ above preconditions the boundary data to be compatible with the interior solenoidality. 
The following is proved in \cite[Prop. 3.3]{MuSc} and it provides a solenoidal extension. For that we introduce the solenoidal space $W^{1,1}_{\Div}(\mathcal O\cup S_{\alpha} ):=\{\mathbf w\in W^{1,1}(\mathcal O\cup S_{\alpha} )\, :\, \divx\mathbf w=0\}$ for $0<\alpha<L$. 
\begin{proposition}
\label{prop:musc}
For a given $\eta\in L^\infty(I;W^{1,2}(\Gamma))$ with $\|\eta\|_{L^\infty(I\times \Gamma)}<\alpha<L$, there is a linear operator
\begin{align*}
 \test :\{\xi\in L^1(I;W^{1,1}(\Gamma)):\,\mathscr K_\eta(\xi)=0\}\rightarrow L^1(I;W^{1,1}_{\Div}(\mathcal O\cup S_{\alpha} )),
\end{align*}
such that the tuple $( \test (\xi-\mathscr K_\eta(\xi)),\xi-\mathscr K_\eta(\xi))$ satisfies
\begin{align*}
 \test (\xi-\mathscr K_\eta(\xi))&\in L^\infty(I;L^2({\mathcal{O}}_\eta))\cap L^2(I;W^{1,2}_{\Div}({\mathcal{O}}_\eta)),\\
\xi-\mathscr K_\eta(\xi)&\in L^\infty(I;W^{2,2}( \Gamma ))\cap  W^{1,\infty}(I;L^{2}( \Gamma )),
\\
( \test &(\xi-\mathscr K_\eta(\xi))\circ\bm{\varphi}_\eta=\bn(\xi-\mathscr K_\eta(\xi)),
\\
\partial_t( \test &(\xi-\mathscr K_\eta(\xi))  \in L^2(I;L^2({\mathcal{O}}_\eta)),
\\
 \test (\xi-\mathscr K_\eta&(\xi))(t,\bx)=0 \text{ for } (t,\bx)\in I \times ({\mathcal{O}} \setminus S_{\alpha})
\end{align*}
provided we have $\xi\in L^\infty(I;W^{2,2}( \Gamma ))\cap  W^{1,\infty}(I;L^{2}(\Gamma))$.
In particular, we have the estimates
\begin{align}
\label{musc0}
\| \test (\xi-\mathscr K_\eta(\xi))\|_{L^q(I;L^{p}({\mathcal{O}} \cup S_{\alpha}  ))}\lesssim \|\xi\|_{L^q(I;L^{p}( \Gamma ))},\\
\label{musc1}
\| \test (\xi-\mathscr K_\eta(\xi))\|_{L^q(I;W^{1,p}({\mathcal{O}} \cup S_{\alpha}  ))}\lesssim \|\xi\|_{L^q(I;W^{1,p}( \Gamma ))}+\|\xi\naby \eta\|_{L^q(I;L^{p}( \Gamma ))},\\
\label{musc2}\|\partial_t \test (\xi-\mathscr K_\eta(\xi))\|_{L^q(I;L^{p}( {\mathcal{O}}\cup S_{\alpha}))}\lesssim \|\partial_t\xi\|_{L^q(I;L^{p}( \Gamma ))}+\|\xi\partial_t \eta\|_{L^q(I;L^{p}( \Gamma ))},
\end{align}
for any $p\in (1,\infty),q\in[1,\infty]$.
\end{proposition}
The following result is a consequence of Proposition \ref{prop:musc}. 
\begin{corollary}\label{cor:3.5}
Let the assumptions of Proposition \ref{prop:musc} be satisfied and in addition, let $a,r\in[2,\infty]$, $p,q\in(1,\infty)$ and $s\in[0,1]$, and assume that $\eta\in L^r(I;W^{2,a}(\Gamma))\cap W^{1,r}(I;L^a(\Gamma))$. Let $\xi \in W^{s,p}(\Gamma)$ and let $\xi_\delta$ be a smooth approximation of $\xi$ in $\Gamma$. Then $\mathcal E^\eta_{\delta}(\xi):=\test (\xi_\delta-\mathscr K_\eta(\xi_\delta))$ satisfies all the conclusions in Proposition \ref{prop:musc}. In particular,
\begin{align*}
\Vert \partial_t \mathcal E^{\eta}_\delta(\xi) \Vert_{L^r(I;L^a(\mathcal{O}\cup S_\alpha))}
\lesssim
\Vert (\xi_\delta)\partial_t\eta \Vert_{L^r(I;L^a(\Gamma))}
\end{align*}
and
\begin{align*}
\Vert \mathcal E^{\eta}_\delta(\xi) - \test(\xi-\mathscr K_\eta(\xi))\Vert_{L^p(\mathcal{O}\cup S_\alpha)}
\lesssim
\Vert \xi_\delta - \xi \Vert_{L^p(\Gamma)}
\end{align*}
holds uniformly in $t\in I$.
\end{corollary}

\section{Stochastic framework and the main result}
\label{sec:frame}

\subsection{Stratonovich-to-It\^o map}
\label{sec:strat}
To properly introduce the stochastic fluid-structure interaction model, we first consider a disjoint partition of the set of punctured integers $\mathbb{Z}\setminus\{0\} $ into the union of its positive $\mathbb{Z}_{+}$ and negative $\mathbb{Z}_{-}$ parts and define the family $(\sigma_k)_{k\in \mathbb{Z}}$ of complex coefficients as
\begin{equation}\label{ansatz}
\left\{\begin{aligned}
\sigma_k(y) & = \frac{ e^{\mathrm{i} ky}}{|k|^{1/2} }\bm{1}_{\{N\leq|k|\leq N^2\}}
\Big(\sum_{N\leq|k|\leq N^2}\frac{1}{|k|}\Big)^{-1/2}, \qquad N\in\mathbb{N}, 
\\[0.4em]
\sigma_{0} (y) & = 0.
\end{aligned} \right.
\end{equation}
Note that $\sigma_k$ is compressible for $k\neq 0$ since in general $\partial_y\sigma_k\neq0$.
Let $(\Omega,\mathfrak F,(\mathfrak F_t)_{t\geq0},\mathbb P)$ be a stochastic basis with a complete, right-continuous filtration. We pair the constant coefficients above with the family $(W_t^k)_{t\geq0}$ of complexified Brownian motions relative to the filtration $(\mathfrak F_t)$ and defined by
\begin{align}\label{BMimaginary}
W_t^k
= \left\{
  \begin{array}{lr} 
 B_t^k+i \,B_t^{-k} & \text{ if }k\in \mathbb{Z}_{+}
 \\ 
B_t^{-k}-i \,B_t^{k} & \text{ if }k\in \mathbb{Z}_{-}
  \end{array}.
\right.
\end{align}
Here, $(B_t^k)_{t\geq0}$ is the usual family of real-valued, independent, identically distributed Brownian motions leading to
\begin{align}\label{qv}
\left\langle W^k,W^\mathbf{\ell}\right\rangle_t
=  
2t\,\delta_{k,-\mathbf{\ell}}.
\end{align}  
We note that due to the precise definition of the noise coefficient, 
we obtain the following Stratonovich-to-It\^o map
\begin{align*} 
 \sqrt{\gamma}\sum_{k\in\mathbb{Z}}  \sigma_k  \,\naby\partial_t\eta   
  \circ  \dd W_t^k 
=&   
\sqrt{\gamma}
\sum_{k\in\mathbb{Z}} \sigma_k  \,\naby\partial_t\eta  
  \dd W_t^k 
 +
 \frac{1}{2} 
 \dd
 \left\langle  
 \sum_{k\in\mathbb{Z}} \sqrt{\gamma}  \sigma_k  \,\naby\partial_t\eta  
 , W^k \right\rangle_t.
\end{align*} 
However, since  
\begin{align*}
\dd(\sqrt{\gamma}\,\sigma_k  \, \naby\partial_t\eta )
=& 
-
\sqrt{\gamma}\,\sigma_k  \,\naby(\naby^4 \eta  + g_\eta -g)\dt
+
\gamma
\sum_{\ell\in\mathbb{Z}}
\big( \sigma_k\sigma_\ell  \,\naby^2\partial_t\eta  
+
\sigma_k(\naby  \sigma_\ell)\naby\partial_t\eta  
\big)\circ\dd W_t^\ell
\end{align*}
where $g_\eta$ is defined right after \eqref{divfree} below,
it follows from \eqref{qv} and the identity $\sigma_{-k}=\overline{\sigma}_k$ that 
\begin{align*}
 \frac{1}{2} 
 \dd 
 \left\langle  
 \sum_{k\in\mathbb{Z}} \sqrt{\gamma}  \sigma_k  \,\naby\partial_t\eta  
 , W^k \right\rangle_t
 &=
 \frac{\gamma}{2} \sum_{k\in\mathbb{Z}} \sum_{\ell\in\mathbb{Z}}\big( \sigma_k\sigma_\ell  \,\naby^2\partial_t\eta  
+
\sigma_k(\naby  \sigma_\ell)\naby\partial_t\eta  
 \big) \dd \langle W^k, W^\ell \rangle_s
\\&=
 \frac{\gamma}{2} \sum_{k\in\mathbb{Z}} \sum_{\ell\in\mathbb{Z}}\big( \sigma_k\sigma_\ell  \,\naby^2\partial_t\eta  
+
\sigma_k(\naby  \sigma_\ell)\naby\partial_t\eta  
\big)  \, 2\,\delta_{k,-\ell}\dd t
\\&=
 \gamma \sum_{k\in\mathbb{Z}}  \Big( |\sigma_k|^2\naby^2\partial_t\eta  
+ 
\sigma_k(\naby \overline{\sigma}_k) 
 \naby\partial_t\eta   
\Big)  \dd t
\end{align*} 
Since
\begin{align*} 
 \sum_{k\in\mathbb{Z}}
 |\sigma_k|^2
=1
\end{align*}
and
\begin{align*}  
 \sum_{k\in\mathbb{Z}}\sigma_k(\naby \overline{\sigma}_k) 
 =
 \bigg(
 \sum_{-N^2\leq k \leq -N}
  \frac{ -\mathrm{i} k}{|k|}
  +
   \sum_{N\leq k\leq N^2}
  \frac{ -\mathrm{i} k}{|k|}
  \bigg)
\bigg(\sum_{N\leq|k|\leq N^2}\frac{1}{|k|}\bigg)^{-1}
=0,
\end{align*}
we can conclude that
\begin{align*} 
 \sqrt{\gamma}\sum_{k\in\mathbb{Z}}  \sigma_k  \,\naby\partial_t\eta  
   \circ  \dd W_t^k 
=&   
\sqrt{\gamma}
\sum_{k\in\mathbb{Z}}  \sigma_k  \,\naby\partial_t\eta  
   \dd W_t^k 
 +
 \gamma  \naby^2\partial_t\eta  
   \dd t. 
\end{align*}  
Thus, it follows that  we can rewrite the equation \eqref{eq:sfsi} for $\eta$ as
 \begin{equation}
\begin{aligned}\label{sde2}
\dd &\partial_t \eta + \Big(\naby^4 \eta   - (1+\gamma) \naby^2 \partial_t\eta + g_\eta-g 
 \Big)\dt
=
\sqrt{\gamma}
\sum_{k\in\mathbb{Z}}   {\sigma}_k \naby\partial_t\eta  \dd W_t^k.
\end{aligned}  
\end{equation}  
%
%

\subsection{The stochastic system}
\label{sec:stoch}
Given the above preparation, we consider the  model for the flow of an incompressible viscous fluid interacting with a flexible structure with transport noise.
In particular, for a given pair of forces $g:(t, \by)\in  I \times \Gamma \mapsto   g(t,\by)\in \mathbb{R}$
 and $\bff:(t, \bx)\in  I \times \Oeta \mapsto   \bff(t,\bx)\in \mathbb{R}^2$,
 we consider the problem where the unknowns consists of a structure displacement function $\eta:(\omega,t, \by)\in \Omega\times I \times \Gamma \mapsto   \eta(\omega,t,\by)\in \mathbb{R}$, a fluid velocity field $\mathbf{u}:(\omega,t, \mathbf{x})\in \Omega\times I \times \Oeta \mapsto  \mathbf{u}(\omega,t, \mathbf{x}) \in \mathbb{R}^2$ and a pressure function $p:(\omega,t, \mathbf{x})\in \Omega\times I \times \Oeta \mapsto  p(\omega,t, \mathbf{x}) \in \mathbb{R}$ 
 such that the system of equations 
\begin{equation}
\begin{aligned}\label{divfree}  
&\divx \bu=0 & \text{in } I\times\Oeta, 
\\
&\partial_t \bu  + (\mathbf{u}\cdot \nabx)\mathbf{u} 
= 
 \delx \bu -\nabx p +\bff& \text{in } I\times\Oeta,  
\\
&\dd \partial_t \eta + ( \naby^4 \eta - (1+\gamma)\partial_y^2\partial_t\eta  + g_\eta -g)\dt
=
\sqrt{\gamma}
\sum_{k\in\mathbb{Z}}   {\sigma}_k \naby\partial_t\eta\dd W_t^k
& \text{in } I\times\Gamma
,  
\\&\bu\circ\bm{\varphi}_\eta=(\partial_t\eta)\bn & \text{on } I\times \Gamma,
\\
&\eta(0,\cdot)=\eta_0(\cdot),  & \text{on } \Gamma,
\\
&\partial_t\eta(0,\cdot)=\eta_\star(\cdot) & \text{on } \Gamma,
\\
&\bu(0,\cdot)=\bu_0(\cdot) & \text{in } \mathcal{O}_{\eta_0}
\end{aligned}
\end{equation}
holds.  Here,
\begin{align*}
g_\eta:=( \mathbb{S}\circ \bm{\varphi}_\eta\bn_\eta )\cdot\bn \,|\naby\bm{\varphi}_\eta|,
\qquad
\mathbb{S}:= (\nabx \bu +(\nabx \bu)^\top) -p\mathbb{I} , 
\end{align*}
$\bn$ and $\bn_\eta$ are the  normal vectors to $\Gamma$ and $\partial\Oeta$, respectively, and $\mathbb{I}$ is the $2\times2$ identity matrix. Our   interest is in solutions of \eqref{divfree} that are both weak in the deterministic and stochastic sense. This refers to the so-called \textit{weak martingale solutions}. In order to give the precise formulation of this solution, we also introduce the Piola transform
\begin{align}
\label{piolaTransform}
\mathcal{J}_{\zeta}\bv
=
\big(\nabx  \bm{\Psi}_{\zeta}(\mathrm{det}\nabx  \bm{\Psi}_{\zeta})^{-1}
\bv
\big)\circ \bm{\Psi}_{\zeta}^{-1}
\end{align}
of a vector field $\mathbf{v}:\mathcal O\rightarrow\R^2$  with respect to a mapping $\zeta:\Gamma \rightarrow \mathbb{R}$. The Piola transform is invertible with inverse
\begin{align}
\label{piolaTransformInverse}
\mathcal{J}_{\zeta}^{-1}\bv
=
\big((\nabx  \bm{\Psi}_{\zeta})^{-1}(\mathrm{det}\nabx  \bm{\Psi}_{\zeta}) 
\bv
\big)\circ \bm{\Psi}_{\zeta} 
.
\end{align}
Compared to the Hanzawa transform $\bfPsi_\zeta$ given in \eqref{map},
it preserves in addition to vanishing boundary values also the divergence-free property of a function.
In order to compensate for the additional factor $(\mathrm{det}\nabx  \bm{\Psi}_{\zeta})^{-1}
$ in the trace of $\mathcal{J}_{\zeta}\bv$, we define the mapping
\begin{align*}
\iota_\zeta \phi:=(\mathrm{det}\nabx  \bm{\Psi}_{\zeta}
)^{-1}
\phi
\end{align*}
for a function $\phi:\Gamma\rightarrow\R$. If $\bm{\phi}\circ \bm{\varphi}=\phi\bn$ on $\Gamma$ it follows that $(\mathcal J_\eta\bfphi)\circ \bm{\varphi}_\eta=\iota_\eta\phi\bn$ on $\Gamma$.
Thus a pair of test-functions $(\phi,\bfphi)$ on the reference domain with the correct boundary condition leads to a pair of test-functions $(\iota_\eta\phi,\mathcal J_{\eta(t)} \bm{\phi})$ on the moving domain with the right boundary condition. Also, a dense set of test-functions on the reference domain  leads to a dense set of test-functions on the moving domain, see \cite[page 237]{lengeler2014weak}, provided $\eta$ is sufficiently smooth. We can now proceed to state the precise definition of a solution of \eqref{divfree}.
\begin{definition}[Weak martingale solution]
\label{def:weakMartSolution}
Let $T>0$ and  $(\sigma_k)_{k\in \mathbb{Z}}$ be given by \eqref{ansatz}.
Let $(\bff, g,\eta_0, \eta_\star,  \bu_0)$ be a dataset such that
\begin{equation}
\begin{aligned}
\label{datasetStrongSol}
&\bff\in L^2(I;L^2_{\mathrm{loc}}(\mathbb{R}^2)), \quad g\in L^2(I;L^2(\Gamma)),
\\&
\eta_0 \in W^{2,2}(\Gamma ) \text{ with } \Vert \eta_0 \Vert_{L^\infty( \Gamma )} < L, \quad
\eta_\star \in L^{2}(\Gamma ),  
\\
&\bu_0\in L^{2}_{\mathrm{\divx}}(\mathcal{O}_{\eta_0} ) \text{ is such that }
 \bu_0 \circ\bm{\varphi}_{\eta_0}   =\eta_\star \bn \text{ on } \Omega\times \Gamma.
\end{aligned}
\end{equation} 
We call $((\Omega,\mathfrak{F},(\mathfrak{F}_t)_{t\geq0},\mathbb{P}),\eta,\bu,(W^k)_{k\in\mathbb Z})$
a \textit{weak martingale solution}  of \eqref{divfree} with  data $(\bff,g,\eta_0, \eta_\star,  \bu_0)$ provided that the following holds:
\begin{itemize}
\item[(a)]  $(\Omega,\mathfrak{F},(\mathfrak{F}_t)_{t\geq0},\mathbb{P})$  is a stochastic basis;
\item[(b)] $(W^k)_{k\in\mathbb Z}$ is a sequence of $(\mathfrak{F}_t)_{t\geq0}$-Brownian motions with \eqref{qv};
\item[(c)] $ \eta$ is $(\mathfrak{F}_t)$-adapted with 
\begin{align*}
&\eta \in L^{\infty} (I;W^{2,2}(\Gamma ))\cap
L^{2} (I;W^{2+s,2}(\Gamma ))\cap W^{1,2} (I;W^{1,2}(\Gamma ))
,
\qquad
\partial_t\eta  \in  C_w (\overline{I};L^{2}(\Gamma ))  \qquad \text{a.s.}
\end{align*} 
for all $s\in(0, 1/2)$;
\item[(d)] the velocity $\bu$ is $(\mathfrak{F}_t)$-adapted with $\bu \circ\bm{\varphi}_{\eta}   =\partial_t \eta  \bn \text{ on } I\times \Gamma$ and
\begin{align*}
&\bu \in L^{2} (I;W^{1,2}(\Oeta ))
\cap  C_w (\overline{I};L^{2}_{\divx}(\Oeta ))  \qquad \text{a.s.};
\end{align*} 
\item[(e)] the equation\footnote{See Section \ref{sec:strat} above for the Stratonovich-to-It\^o map.}
\begin{align*}
\dd\bigg(\int_{\Gamma}\partial_t\eta& \,\iota_\eta\phi\dy
+
\int_{\Oeta}\bu\cdot\mathcal{J}_\eta\bm{\phi}\dx
\bigg)
\\&=
\int_{\Gamma}\big( \partial_t\eta \, \partial_t(\iota_\eta\phi )
-(1+\gamma)\partial_t\partial_y\overline{\eta}\,\partial_y (\iota_\eta\phi)
-
 \naby^2\eta \,\naby^2(\iota_\eta\phi )
 +g\,\iota_\eta\phi
\big)\dy\dt
\\
&+\int_{\Oeta}\big(\bu\cdot\partial_t(\mathcal{J}_\eta\bm{\phi}
)+
 \bu\otimes \bu:\nabx\mathcal{J}_\eta\bm{\phi}
-
\nabx\bu:\nabx\mathcal{J}_\eta\bm{\phi}+\bff\cdot\mathcal{J}_\eta\bm{\phi}\big)\dx\dt
\\
&
-
\sqrt{\gamma}
 \sum_{k\in\mathbb{Z}}  \int_{\Gamma}\partial_t\eta \sigma_k\,\naby (\iota_\eta\phi )
\dy\dd W_s^k
\end{align*} 
holds a.s. for all $(\phi,\bm{\phi})\in W^{2,2}(\Gamma)\times W^{1,2}_{\divx}(\mathcal{O})$ with $\bm{\phi}\circ\bm{\varphi}=\phi\bn$  on $\Gamma$;
\item[(f)] the inequality
\begin{align*} 
 \Vert\partial_t\eta (t)\Vert_{L^2(\Gamma)}^2&+ \Vert \naby^2\eta(t)\Vert_{L^2(\Gamma)}^2
+\Vert\bu(t)\Vert_{L^2(\Oeta)}^2
+ \int_0^t\Vert\partial_t\partial_y\eta  \Vert_{L^2(\Gamma)}^2\ds + \int_0^t \Vert \nabx\bu \Vert_{L^2(\Oeta)}^2 \ds
\\
\lesssim&
 \Vert \eta_\star \Vert_{L^2(\Gamma)}^2+  \Vert \naby^2\eta_0 \Vert_{L^2(\Gamma)}^2
+\Vert\bu_0\Vert_{L^2(\mathcal{O}_{\eta_0})}^2
\\&+
 \int_0^t \Vert g \Vert_{L^2(\Gamma)}^2 \ds
 +
  \int_0^t \Vert \bff\Vert_{L^2(\Oeta)}^2 \ds\\
&
+
\frac{\sqrt{\gamma}}{2}
 \sum_{k\in\mathbb{Z}}  \int_{\Gamma}|\partial_t\eta|^2\,\naby  \sigma_k
\dy\dd W_s^k.
\end{align*}
holds a.s. for a.a. $t\in I$
\end{itemize}
\end{definition}  
%
The existence of weak martigale solution for the system \eqref{divfree} in the sense of Definition \ref{def:weakMartSolution} can be shown following the arguments in \cite{BrMeMo}.
\begin{theorem}\label{thm:mainA}
Let $(\eta_0, \eta_1, \bu_0,\sigma)$ be a dataset such that
\eqref{datasetStrongSol} holds. Then there is a weak martingale solution of \eqref{divfree}
with data $(\eta_0, \eta_1, \bu_0,\sigma)$ in the sense of Definition \ref{def:weakMartSolution}.
The interval of existence is of the form $\overline{I}=(0,t)$, where $t<T$
 only if $\lim_{s\rightarrow t}\|\eta(s)\|_{L^\infty(\Gamma)}=L$
a.s. in $\Omega_0$ for some $\Omega_0\subset\Omega$ with $\mathbb P(\Omega_0)>0$.
\end{theorem}
\begin{proof}
The proof follows the strategy from \cite{BrMeMo} where the interaction of a three-dimensional fluid with a two dimensional perfectly elastic structure is considered. First of all,
the dissipation in the structure does not alter the arguments but gives the additional information
$\eta\in W^{1,2} (I;W^{1,2}(\Gamma ))$ $\mathbb P$-a.s.
as well as the additional dissipative term on the left-hand side of the energy balance.
The main difference is the additional martingale
\begin{align*}
\mathfrak M_t^N:=
\frac{\sqrt{\gamma}}{2}
 \sum_{k\in\mathbb{Z}}  \int_{\Gamma}|\partial_t\eta|^2 \partial_y\sigma_k\,
\dy\dd W_s^k
\end{align*} 
appearing on the right-hand side of the energy balance. The corresponding term in \cite{BrMeMo} vanishes by the solenoidability of the $\sigma_k$'s assumed there leading to a pathwise energy bound. Note that for fixed $N$ taking the derivative on $\sigma_k$ is harmless and only changes in the limit $N\rightarrow \infty$. 
Given the estimate
\begin{align*}
\Vert f^2\Vert_{W^{1/2,2}(\Gamma)} ^2
&\lesssim
\Vert f\Vert_{L^{\infty}(\Gamma)}^2 
\Vert f\Vert_{W^{1/2,2}(\Gamma)}^2 
\\&\lesssim
\Vert f\Vert_{L^{2}(\Gamma)}^2 \Vert f\Vert_{W^{1,2}(\Gamma)}^2  
\end{align*}
it follows from using the Burkholder-Davis-Gundi inequality that
\begin{align}
\mathbb E\bigg[\sup_{t\in I}\big|\mathfrak M_t^N\big|\bigg]
&\lesssim \mathbb E\bigg[\bigg(\sum_{k\in\mathbb Z} \int_I\Big(\int_{\Gamma}\partial_y^{1/2}|\partial_t\eta|^2  \partial_y^{1/2}\sigma_k\,
\dy\Big)^2\dt\bigg)^{1/2}\bigg]\nonumber
\\&\lesssim
\frac{1}{\sqrt{\log(N)}} \mathbb E\bigg[\bigg(\sum_{k\in\mathbb Z} \int_I\Big(\int_{\Gamma}\partial_y^{1/2}|\partial_t\eta|^2  e^{iky}\,
\dy\Big)^2\dt\bigg)^{1/2}\bigg]\nonumber
\\&\lesssim
\frac{1}{\sqrt{\log(N)}} \mathbb E\bigg[\bigg(  \int_I\big\Vert |\partial_t\eta|^2  \big\Vert_{W^{1/2,2}(\Gamma)}^2\dt\bigg)^{1/2}\bigg] \nonumber
\\
&\lesssim 
\frac{1}{\sqrt{\log(N)}}\mathbb E\bigg[\sup_{t\in I}\|\partial_t\eta\|_{L^2(\Gamma)} \Big(\int_I\|\naby\partial_t\eta\|_{L^2(\Gamma)}^2\dt\Big)^{1/2}\bigg]\nonumber
\\
&\lesssim\frac{1}{\sqrt{\log(N)}}\mathbb E\bigg[\sup_{t\in I}\|\partial_t\eta\|_{L^2(\Gamma)}^2+\int_I\|\naby\partial_t\eta\|_{L^2(\Gamma)}^2\dt\bigg]\label{new}
\end{align}
holds and can be absorbed into the left-hand side for sufficiently large $N$.
A similar procedure can be implemented for arbitrary high moments of the energy.
This is sufficient to apply the stochastic compactness argument in \cite{BrMeMo} as the remaining estimates do not use the soilenoidability of the noise.
Note that one can appy \cite[Theorem 2.9.1]{BFHbook}
to show that the energy inequality continues to hold on the new probability space.
Also note that the stochastic compactness methods yields compactness of $\partial_t\eta$ in $L^2(I\times\Gamma)$. This justifies the limit procedure in the energy inequality.
\end{proof}

\subsection{The main result}
We are finally in the position to present the statement of the  main result of this paper.

\begin{theorem}\label{thm:main}
Let $\kappa\in(0,1)$ and $T>0$ be given and assume that the dataset $(\bff, g,\eta_0, \eta_\star,  \bu_0)$ satisfies \eqref{dataset} and \eqref{dataset'}. Then there are $\gamma>0$ and $N\in\mathbb N$ such that any weak martingale solution $((\Omega,\mathfrak{F},(\mathfrak{F}_t)_{t\geq0},\mathbb{P}),\eta^N,\bu^N,(W^k)_{|k|\leq N})$ of \eqref{divfree} in the sense of Definition \ref{def:weakMartSolution} with data $(\bff,g,\eta_0, \eta_\star,  \bu_0)$ satisfies
\begin{align*}
{\mathbb{P}}\big(\Vert \eta^N\Vert_{L^\infty((0,T)\times\Gamma)}<L\big)\geq 1-\kappa,
\end{align*}
{\emph i.e.}, with probability $1-\kappa$ there is no contact until the time $T$.
\end{theorem}
The proof of Theorem \ref{thm:main} can be found in  the subsequent two sections and will be completed in Section \ref{sec:no}. In a first step (Section \ref{sec:comp}) we prove that the sequence $((\Omega,\mathfrak{F},(\mathfrak{F}_t)_{t\geq0},\mathbb{P}),\eta^N,\bu^N,(W^k)_{|k|\leq N})$ converges (after taking a subsequence) in law to a solution of the deterministic problem. This is reminiscent of our analysis in \cite{BrMeMo}.
In a second step we strengthen this to convergence in probability on the original probability space (which can be assumed to be independent of $N$), see Section \ref{sec:pcon}.
Crucial tools here are, the weak strong-uniqueness property of the deterministic problem, higher differentiability of the structure function, and 
a version of the Gy\H{o}ngy-Krylov characterisation of convergence in probability.

\begin{remark}
It is currently open if strong solutions to the stochastic fluid-structure interaction problem \eqref{divfree} exist even if additional dissipation in the structure is allowed.
The corresponding deterministic results, cf.~\cite{breit2022regularity,BMSS, MR3466847, ScSu} are based on temporal derivatives. Such an approach cannot be replicated in the stochastic setting. Hence it is unclear if a stochastic counterpart can be achieved.
Consequently, our main result is formulated for martingale solutions.
\end{remark}
%
%

\section{Stochastic compactness}
\label{sec:comp}
Let $((\Omega,\mathfrak{F},(\mathfrak{F}_t)_{t\geq0},\mathbb{P}),\eta^N,\bu^N,(W^k)_{|k|\leq N})$ be a weak martingale solution of \eqref{divfree} in the sense of Definition \ref{def:weakMartSolution} with data $(\bff,g,\eta_0, \eta_\star,  \bu_0 )$. Then it satisfies the following energy estimate 
\begin{equation}
\begin{aligned}
\label{energyGal}
\mathbb E\bigg[
\sup_{t\in I}\Big( \Vert\partial_t\eta ^N(t)\Vert_{L^2(\Gamma)}^2&+  \Vert \naby^2\eta^N(t)\Vert_{L^2(\Gamma)}^2
+
\Vert \bu^N(t) \Vert_{L^2(\mathcal{O}_{\eta^N})}^2
\Big)\\
&+ \int_I\Vert\partial_t\partial_y\eta\Vert_{L^2(\Gamma)}^2\dt
+ 
 \int_I \Vert \nabx\bu^N \Vert_{L^2(\mathcal{O}_{\eta^N})}^2
\dt
\bigg]
\\
\lesssim&
 \Vert \eta_\star \Vert_{L^2(\Gamma)}^2+  \Vert \naby^2\eta_0 \Vert_{L^2(\Gamma)}^2
+\Vert\bu_0\Vert_{L^2(\mathcal{O}_{\eta_0})}^2
\\&+
 \int_I \Vert g \Vert_{L^2(\Gamma)}^2 \dt
 +
  \int_I\Vert \bff\Vert_{L^2(\mathcal{O}_{\eta^N})}^2 \dt
\end{aligned}
\end{equation}
uniformly in $N\in\mathbb{N}$. From this estimate, it follows from the interface condition $\bu^N\circ\bm{\varphi}_{\eta^N } =\bn\partial_t\eta^N $, \eqref{energyGal} and the trace theorem that
\begin{equation}
\begin{aligned}
\label{GalerkinConvEta2Np}
\mathbb E\bigg[
\int_I\Vert \partial_t\eta^N \Vert_{W^{1/2,2}(\Gamma)}^2\dt
\bigg]
\lesssim 1
\end{aligned}
\end{equation}
holds.  Note $\partial\mathcal{O}_{\eta^N}$ is Lipschitz continuous uniformly in time and $N$ because of \eqref{energyGal} and the one dimensional Sobolev embedding $W^{2,2}(\Gamma)\hookrightarrow W^{1,\infty}(\Gamma)$. Based on these estimates we aim to pass to the limit in the coupled momentum-structure equation. Therefore we apply the stochastic compactness method.
\subsection{Tightness}
Due to the moving boundary of the fluid's domain
the tightness of the velocity field in a strong topology is very delicate and will be ultimately achieved by Proposition \ref{prop:musc} using the solenoidal extension $\mathscr F^{\eta^N}$. Following \cite{lengeler2014weak}, we write
\begin{align}\label{eq:convrhouN}\begin{aligned}
\int_{I}\int_{\mathcal O_{\eta^N}}|\bu^N|^2\dxt&
+\int_{I}\int_\Gamma|\partial_t\eta^N|^2\,\dy \dt
\\&
=\int_{I}\int_{\mathcal O\cup S_{\alpha}}\mathbb I_{\mathcal O_{\eta^N}}
\bu^N\cdot\mathscr F^{\eta^N}( \partial_t\mathcal \eta^N)\dxt+\int_{I}\int_\Gamma|\partial_t\eta^N|^2\,\dy \dt
\\
&+\int_{I}\int_{\mathcal O\cup S_{\alpha}}
\mathbb I_{\mathcal O_{\eta^N}}
\bu^N\cdot(\bu^N-\mathscr F^{\eta^N}(\partial_t\mathcal \eta^N))\dxt 
\end{aligned}
\end{align} 
and consider the space $X:=L^2(\Gamma)\times W^{-s,2}(\mathcal{O}\cup S_{\alpha})$ with  $s\in(0,1/4)$.  
In order to apply Theorem~\ref{thm:auba} yielding tightness of the corresponding laws,
we need to equip $L^2(I;X\times X')$ with an unconventional topology which we denote by  $\tau_\sharp$
and define the convergence $\rightarrow^\sharp$  as follows. 
 We say that
$$((\zeta^N, \bv^N),(\xi^N,\mathbf{w}^N))
\rightarrow^\sharp((\zeta, \mathbf v),(\xi, {\mathbf w}))\,\,\text{in}\,\,L^2(I;X\times X')$$ provided
that
$$((\zeta^N, \bv^N),(\xi^N,\mathbf{w}^N))
 \rightharpoonup ((\zeta, \mathbf v),(\xi, {\mathbf w}))\,\,\text{in}\,\,L^2(I;X\times X')$$
and it holds
\begin{align}\label{eq:tau}
\begin{aligned} 
\int_I\langle \bv^N,\mathbf{w}^N\rangle_{W^{-s,2},W^{s,2}}\dt
&+\int_I\int_\Gamma\zeta^N\,\xi^N\dy\dt
\\
&\longrightarrow \int_I\langle \mathbf v,{\mathbf w}\rangle_{W^{-s,2},W^{s,2}}\dt
+\int_I\int_\Gamma\zeta\,\xi\dy\dt.
\end{aligned}
\end{align}
Since this topology is finer than the weak topology on $L^2(I;X\times X')$, it is clear
 that $(L^2(I;X\times X'),\tau_\sharp)$ is a quasi-Polish
space such that Jakubowski's version of the Skorokhod representation theorem
applies. Now we prove the required tightness for $(\eta^N,\bu^N)$. 
\begin{lemma}\label{lemma:fg'}
The laws of
\begin{align*}
\big((\partial_t\eta^N,\mathscr F^{\eta^N}(\partial_t\eta^N)),(\partial_t\eta^N,\mathbb I_{\mathcal O_{\eta^N}}\mathbf u^N)\big)\quad\text{and}\quad
\big((0,\mathbf u^N-\mathscr F^{\eta^N}(\partial_t\eta^N)), (\partial_t\eta^n,\mathbb I_{\mathcal O_{\eta^N}}\mathbf u^N)
\big)
\end{align*}
on $(L^2(I;X\times X'),\tau_\sharp)$ are tight.
\end{lemma}
\begin{proof}
In both cases we must verify the assumptions (a)--(d) from Theorem \ref{thm:auba}, where we choose $\mathfrak s=2$, $X:=L^2(\Gamma)\times W^{-s,2}(\mathcal{O}\cup S_{\alpha})$ for $0<s<1/4$ and $Z:=W^{s_0,2}(\Gamma)\times W^{s_0,2}(\mathcal{O}\cup S_{\alpha})$ with $s<s_0<1/4$.

\textbf{Tightness for $(\partial_t\eta^N,\mathscr F^{\eta^N}(\partial_t\eta^N))$.} 
Assumption~$(a)$.
 First of all, boundedness of
\begin{align*}
(f_N,g_N):=((\partial_t\eta^N,\mathscr F^{\eta^N}(\partial_t\eta^N)),(\partial_t\eta^N,\mathbb I_{\mathcal{O}_{\eta^N}}\bu^N))
\end{align*}
 in $L^2(I;X\times X')$
 follows from \eqref{energyGal}--\eqref{GalerkinConvEta2Np} and the properties of the extension $\mathscr F^{\eta^N}$ from Proposition \ref{prop:musc}.

Assumption~$(b)$.  
We consider the regularisation of $\partial_t\eta^N$ with parameter $\kappa>0$
 and set for $N\in\mathbb N$
\[
f_{N,\kappa}(t):=((\partial_t\eta^N(t))_\kappa,
\mathcal E_\kappa^{\eta^N}(\partial_t\eta^N(t))),
\]
where $\mathcal E_\kappa^{\eta^N}$ is given in Corollary \ref{cor:3.5}. The regularisation of $\partial_t\eta^N(t)$ can be achieved, for instance, by projecting on finitely many Fourier modes.
We find by the continuity of the mollification operator 
and the continuity of $\mathcal E_\kappa^{\eta^N}$ from 
Corollary \ref{cor:3.5} that for a.e.\ $t\in (0,T)$ and $s\leq 1/2$,
 \begin{align*}
 \|f_{N,\kappa}-f_N\|_{L^2(\Gamma)\times L^{2}(\mathcal{O}\cup S_{\alpha})}
\leq \,c\kappa^{s}\|\partial_t\eta^N\|_{W^{s,2}(\Gamma)},
 \end{align*}
 which can be made arbitrarily small in $L^2(I)$ by choosing $\kappa$ appropriately, 
cf. \eqref{GalerkinConvEta2Np}.
Similarly, we have
 \begin{align*}\label{eq:fNkappa}
 \|f_{N,\kappa}\|_{W^{1,2}(\Gamma)\times W^{1,2}(\mathcal{O}\cup S_{\alpha})}
\leq \,c\kappa^{-1}\|\partial_t\eta^n\|_{L^{2}(\Gamma)}
 \end{align*}
and thus we can choose $Z:=W^{s_0,2}(\Gamma)\times W^{s_0,2}(\mathcal{O}\cup S_{\alpha})$ with $s<s_0<1/4$.
Clearly, if $f_N\rightharpoonup f$ in $L^2(I;X)$ then
 we can deduce a converging subsequence such that $f_{N,\kappa}\weakto f_{\kappa}$ (for some $ f_{\kappa}$) in 
$L^2(I;X)$ for any $\kappa>0$,
 which implies (b).

Assumption~$(c)$. We have to control
$\langle f_{N,\kappa}(t), g_N(t)-g_N(s) \rangle$, 
where $g_N:=(\partial_t\eta^N,\mathbb I_{\mathcal{O}_{\eta^N}}\bu^N)$
and hence decompose 
\begin{align*}
&\langle f_{N,\kappa}(t), g_N(t)-g_N(s)\rangle
\\
&=\Big(\big\langle \big((\partial_t\eta^N(t))_\kappa,
\mathcal E_\kappa^{\eta^N(t)}(\partial_t\eta^N(t))\big),
 g_N(t)
\big\rangle
\\
&\quad-\big\langle \big((\partial_t\eta^N(t))_\kappa,\mathcal E_\kappa^{\eta^N(s)}
(\partial_t\eta^N(t))\big), g_N(s) \big\rangle\Big)
\\
&\quad -\big\langle (0,\mathcal E_\kappa^{\eta^N(t)}(\partial_t\eta^N(t))
-\mathcal E_\kappa^{\eta^N(s)}(\partial_t\eta^N(t)),
g_N(s)
\big\rangle\\
&=:({\tt I})+({\tt II}).
 \end{align*} 
We begin estimating $({\tt II})$ to find that
 \begin{align*}
({\tt II})& = \int_{\mathcal{O}_{\eta^N(s)}} \bu^N(s)
\cdot \int_t^s\partial_r \mathcal E_\kappa^{\eta^N(r)}(\partial_t\eta^N(t))\,\dd r\dx
\\&
\leq\,c\|\bu^N(s)\|_{L^2(\mathcal{O}_{\eta^N(s)})}\|(\partial_t\eta^N)_\kappa\|_{L^\infty(I\times \Gamma)}\|\partial_t\eta^N \|_{L^\infty(I;L^{2}(\Gamma))}|t-s|
\\
&
\leq\,c(\kappa)\|\bu^N(s)\|_{L^2(\mathcal{O}_{\eta^N(s)})}\|\partial_t\eta^N\|_{L^\infty(I;L^2(\Gamma))}^2|t-s|.
 \end{align*}
using Corollary \ref{cor:3.5}. Thus,
by \eqref{energyGal}, we obtain
 \begin{align*}
|({\tt II})|
&
\leq\,c(\kappa)|t-s|.
 \end{align*}
  The term $({\tt I})$ is estimated using Definition \ref{def:weakMartSolution} (e) to obtain 
\begin{equation*}
\begin{aligned}
\int_t^s({\tt I})\,\dd r&=- \int_t^s
\int_{\Gamma} \naby^2\eta^N\, \naby^2 (\partial_t\eta^N(t))_\kappa
\dy\,\dd r\\
&\quad +
\int_t^s \int_{\mathcal{O}_{\eta^N}} \bu^N\cdot \partial_r\mathcal E_\kappa^{\eta^N(r)}
(\partial_t\eta^N(t)) 
\dx\,\dd r
\\
&\quad +
 \int_t^s\int_{\mathcal{O}_{\eta^N}} 
 {\bf u}^N\otimes  \bfu^N:\nabx\mathcal E_\kappa^{\eta^N(r)}
(\partial_t\eta^N(t) )
\dx\,\dd r
\\
&\quad -
 \int_t^s\int_{\mathcal{O}_{\eta^N}}\nabx \bfu^N:\nabx \mathcal E_\kappa^{\eta^N(r)}
(\partial_t\eta^N(t))) \dx\,\dd r
\\&\quad+
 \int_t^s\int_{\mathcal{O}_{\eta^N}}\bff\cdot\mathcal E_\kappa^{\eta^N(r)}
(\partial_t\eta^N(t)) \dx\,\dd r 
+
\int_t^s  \int_{\Gamma}g (\partial_t\eta^N )_\kappa
\dy\dd r
\\
&\quad+ 
 \int_t^s\int_{\Gamma} \partial_t\eta^N  \,
  \naby^2 (\partial_t\eta^N )_\kappa 
\dy\dd r
\\
&\quad+ 
\gamma
 \sum_{k\in\mathbb{Z}} \int_t^s\int_{\Gamma}\partial_t\eta^N  \,
 \sigma_k\,\naby ( \overline{\sigma}_k\,\naby (\partial_t\eta^N )_\kappa)
\dy\dd r
\\&\quad-
\sqrt{\gamma}
 \sum_{k\in\mathbb{Z}}\int_t^s  \int_{\Gamma}\partial_t\eta^N \sigma_k\,\naby (\partial_t\eta^N )_\kappa
\dy\dd W_r^k.
\end{aligned}
\end{equation*}
Due to  Corollary \ref{cor:3.5} and the uniform estimates \eqref{energyGal}--\eqref{GalerkinConvEta2Np},
it is clear that all terms except the last one can be estimated by 
\begin{align*}
({\tt I})&\leq C(\kappa)|t-s|^{1/2}(1+\|\eta^N\|_{C(I\times\Gamma)}+\|\nabx\mathbf u^N\|_{L^2(I\times \mathcal{O}\cup S_\alpha)})\\
&+C(\kappa)|t-s|^{1/2}\|\mathbf u^N\|_{L^2(I\times \mathcal{O}\cup S_\alpha)}^2.
\end{align*}
Note that this estimate is uniform in $N$. Denoting
$$\mathscr M^N_s:=\sqrt{\gamma}\sum_{k\in\mathbb Z}\int_0^s\int_{\Gamma}  
\partial_t\eta^N \sigma_k\,\naby (\partial_t\eta^N )_\kappa
\dy\dd W_t^k$$ 
we obtain by Markov's inequality and It\^o-isometry that for any $K>0$ (using also $\sigma_k$ is constant with $\sum_{k\in\mathbb Z}|\sigma_k|^2=1$),
\begin{align*}
\mathbb P\big(&|t-s|^{-1/2}\big|\mathscr M^N_s-\mathscr M^N_t\big|\geq K\big)
\\
&\leq\,\frac{1}{K^2}\mathbb E\bigg[|t-s|^{-1/2}\bigg|\sqrt{\gamma}\sum_{k\in\mathbb Z}\int_t^s\int_{\Gamma}  
\partial_t\eta^N \sigma_k\,\naby (\partial_t\eta^N )_\kappa \dy\dd W^k_r\bigg|^2\bigg]
\\
&=\,\frac{1}{K^2}\mathbb E\bigg[|t-s|^{-1}\gamma\sum_{k\in\mathbb Z}|\sigma_k|^2\int_t^s\bigg|\int_{\Gamma}  
 \partial_t \eta^N \naby(\partial_t\eta^N(t))_\kappa)\dy\bigg|^2\,\dd r\bigg]
\\
&\lesssim \,\frac{1}{K^2}\mathbb E\bigg[|t-s|^{-1}\int_t^s\|\partial_t\eta^N\|_{L^2(\Gamma)}^2
\|\naby(\partial_t\eta^N)_\kappa\|_{L^2(\Gamma)}^2\,\dd r\bigg]
\\
&\lesssim \,\frac{\kappa^{-1}}{K^2}\mathbb E\bigg[|t-s|^{-1}
\int_t^s\|\partial_t\eta^N\|_{L^2(\Gamma)}^2\|\partial_t\eta^N\|_{L^2(\Gamma)}^2\,\dd r\bigg]\lesssim \,\frac{\kappa^{-1}}{K^2}
\end{align*}
using \eqref{energyGal}.

Assumption~$(d)$. The last property from Theorem \ref{thm:auba} 
follows by the usual compactness in (negative) Sobolev spaces.

We conclude with the tightness of the law of \begin{align*}
\big((\partial_t\eta^N,\mathscr F^{\eta^N}(\partial_t\eta^N)), (\partial_t\eta^N,\mathbb I_{\mathcal O_{\eta^N}}\mathbf u^N)\big)
\end{align*}
on $(L^2(I;X\times X'),\tau_\sharp)$.

\textbf{Tightness for $(0,\mathbf u^N-\mathscr F^{\eta^N}(\partial_t\eta^N))$.} 
The proof follows the lines of the proof of tightness for $(\partial_t\eta^N,\mathscr F^{\eta^N}(\partial_t\eta^N))$ above. The only difference
being the regularisation of
$$f_N:=(0,\mathbf u^N-\mathscr F^{\eta^N}(\partial_t\eta^N)).$$
Assumption~(a). Boundedness of
\begin{align*}
(f_N,g_N):=((0,\bfu^N-\mathscr F^{\eta^N}(\partial_t\eta^N)),(0,\mathbb I_{\mathcal{O}_{\eta^N}}\bu^N))
\end{align*}
 in $L^2(I;X\times X')$
 follows again from \eqref{energyGal}--\eqref{GalerkinConvEta2Np} and the properties of the extension $\mathscr F^{\eta^N}$ from Proposition \ref{prop:musc}.

Assumption $(b)$. We first notice that the function
$$\widetilde{\bu}_N:=\bu^N-\mathscr F^{\eta^N}(\partial_t\eta^N)$$
belongs to $L^2(I;W_{0,\Div}^{1,2}(\mathcal{O}_{\eta^N}))$ by Proposition \ref{prop:musc}. In particular, we have
\begin{align*}
\widetilde{\mathbf u_N} \in L^2(I;L^2_{\Div}(\mathcal{O}_{\eta^N}))
\end{align*}
uniformly in $N$. Now, we first apply an approximation argument with respect to the spatial variable.
%
Since the boundary of $\mathcal{O}_{\eta^N}$ is uniformly Lipschitz continuous by \eqref{energyGal},
we can approximate $\widetilde{\mathbf u}_N(t)$ by $(\widetilde{\mathbf u}_N(t))_\kappa$, where $ (\widetilde{\mathbf u}_N(t))_\kappa \in C^\infty_{c,\Div}((\mathcal{O}_{\eta^N}))$, see \cite[Theorem III.4.1]{MR2808162}. The mapping $\bfw\mapsto (\bfw)_\kappa$ is a combination of mollification and Bogovskii-correction and thus is continuous. Hence the Bochner measurability is preserved.
Also, the regularisation parameter can be chosen such that
\begin{align*}
\|(\widetilde{\mathbf u}_N)_\kappa-\mathbf u^N\|_{W^{-s,2}(\mathcal{O}\cup S_{\alpha})}\lesssim \kappa^s,\quad \|(\widetilde{\mathbf u}_N)_\kappa\|_{W^{1,\infty}(\mathcal{O}\cup S_{\alpha})}\lesssim \|\mathbf u^N_\kappa\|_{W^{3,2}(\mathcal{O}\cup S_{\alpha})}\lesssim \kappa^{-3},
\end{align*}
uniform in time and $N$. Finally, we denote by $\mathbf u^N_\kappa$ a temporal regularisation of $(\widetilde{\mathbf u}_N)_\kappa$ which leads to a function which is also smooth in time. 
 The estimates above, the zero boundary values and the solenoidability remain unchanged by this procedure. In particular, the $W^{1,\infty}_{t,\bx}$-norm can be controlled in terms of $\kappa$ uniform in $N$.
This leads to a function $f_{N,\kappa}:=(0,\mathbf u^N_\kappa)$.

Assumption~$(c)$. The function $f_{N,\kappa}:=(0,\mathbf u^N_\kappa)$ has zero boundary conditions and thus only the fluid part is seen in
the expression $\langle f_{N,\kappa}(t), g_N(t)-g_N(s)\rangle$ 
(where $g_N:=(\partial_t\eta^N,\mathbb I_{\mathcal{O}_{\eta^N}}\mathbf u^N)$
 as above). In particular, we conclude again that
\begin{align*}
\langle f_{N,\kappa}(t), g_N(t)-g_N(s)\rangle&\leq\, C(\kappa)|t-s|^{1/2}(1+\|\eta^N\|_{C(I\times\Gamma)}+\|\nabx\mathbf u^N\|_{L^2(I\times \mathcal{O}\cup S_\alpha)})\\
&+C(\kappa)|t-s|^{1/2}\|\mathbf u^N\|_{L^2(I\times \mathcal{O}\cup S_\alpha)}^2
\end{align*}
uniformly in $N$.

Assumption~$(d)$. We use again compactness in negative Sobolev spaces.

By Theorem \ref{thm:auba} we obtain tightness of the law of
 \begin{align*}
((0,\mathbf u^N-\mathscr F^{\eta^N}(\partial_t\eta^N)), (\partial_t\eta^N,\mathbb I_{\mathcal O_{\eta^N}}\mathbf u^N))
\end{align*}
on $(L^2(I;X\times X'),\tau_\sharp)$.

\end{proof}

\subsection{Compactness}
We define the path space
\begin{align*}
\chi=\chi_W\times \chi_\eta\times \chi_{\partial_t\eta}\times \chi_\bu\times \chi_{\nabx\bu} \times \chi_{f,g}
\end{align*}
where
\begin{align*}
\chi_W&=(C(\overline{I}))^{\mathbb Z}, \quad \chi_\eta\in \big( L^\infty(I,W^{2,2}(\Gamma)),w^*\big)\cap \big( L^\infty(I,L^\infty(\Gamma))\big),\\  \chi_{\partial_t\eta}&= \big( L^\infty(I,L^{2}(\Gamma)),w^*\big)\cap L^2(I;W^{1,2}(\Gamma)),
\\
\chi_\bu&=\big( L^\infty(I,L^{2}(\mathcal{O}\cup S_\alpha)),w^*\big), \quad \chi_{\nabx\bu}\in \big( L^2(I,L^{2}(\mathcal{O}\cup S_\alpha)),w\big)
, \\ \chi_{f,g}&= \big( L^2(I; X\times X'),\tau_\sharp\big).
\end{align*}
With this, the uniform bounds \eqref{energyGal}--\eqref{GalerkinConvEta2Np} and Lemma \ref{lemma:fg'} at hand, we conclude that
\begin{lemma}
The joint law
\begin{align*}
\bigg\{ \mathcal{L}\bigg[(W_t^k)_{k\in\mathbb Z}, \eta^N, \partial_t\eta^N, \mathbb{I}_{\mathcal{O}_{\eta^N}}\bu^N, \mathbb{I}_{\mathcal{O}_{\eta^N}}\nabx\bu^N,
\begin{pmatrix}
((\partial_t\eta^N, \mathscr{F}^{\eta^N}(\partial_t\eta^N)), (\partial_t\eta^N, \mathbb{I}_{\mathcal{O}_{\eta^N}}\bu^N))
\\
( (0,\bu^N- \mathscr{F}^{\eta^N}(\partial_t\eta^N)), (\partial_t\eta^N, \mathbb{I}_{\mathcal{O}_{\eta^N}}\bu^N))
\end{pmatrix}
\bigg]\,;\, N\in \mathbb{N}\bigg\}
\end{align*}
is tight on $\chi$.
\end{lemma}
Thus, by the Jakubowski-Skorokhod representation theorem \cite{jakubow}  (we refer to \cite[Theorem A.1]{on2} for a statement 
which combines Prokhorov's and Skorokhod's theorem for quasi-Polish spaces), we obtain the following.
\begin{proposition}\label{prop:skorokhod}
There exist a complete probability space $(\overline{\Omega},\overline{\mathfrak{F}}, \overline{\mathbb{P}})$ with $\chi$-valued random variables
\begin{align*}
\overline{\bm{\Theta}}^N
:=
\bigg[(\overline{W}_t^{k,N})_{k\in\mathbb Z}, \overline{\eta}^N, \partial_t\overline{\eta}^N, \mathbb{I}_{\mathcal{O}_{\overline{\eta}^N}}\overline{\bu}^N, \mathbb{I}_{\mathcal{O}_{\overline{\eta}^N}}\nabx\overline{\bu}^N,
\begin{pmatrix}
((\partial_t\overline{\eta}^N, \mathscr{F}^{\overline{\eta}^N}(\partial_t\overline{\eta}^N)), (\partial_t\overline{\eta}^N, \mathbb{I}_{\mathcal{O}_{\overline{\eta}^N}}\overline{\bu}^N))
\\
((0,\overline{\bu}^N- \mathscr{F}^{\overline{\eta}^N}(\partial_t\overline{\eta}^N)), (\partial_t\overline{\eta}^N, \mathbb{I}_{\mathcal{O}_{\overline{\eta}^N}}\overline{\bu}^N))
\end{pmatrix}
\bigg] 
\end{align*}
for $N\in\mathbb{N}$ and 
\begin{align*}
\overline{\bm{\Theta}}
:=
\bigg[(\overline{W}_t^k)_{k\in\mathbb Z}, \overline{\eta}, \partial_t\overline{\eta}, \mathbb{I}_{\mathcal{O}_{\overline{\eta}}}\overline{\bu}, \mathbb{I}_{\mathcal{O}_{\overline{\eta}}}\nabx\overline{\bu},
\begin{pmatrix}
((\partial_t\overline{\eta}, \mathscr{F}^{\overline{\eta}}(\partial_t\overline{\eta})), (\partial_t\overline{\eta}, \mathbb{I}_{\mathcal{O}_{\overline{\eta}}}\overline{\bu}))
\\
((0,\overline{\bu}- \mathscr{F}^{\overline{\eta}}(\partial_t\overline{\eta})), (\partial_t\overline{\eta}, \mathbb{I}_{\mathcal{O}_{\overline{\eta}}}\overline{\bu}))
\end{pmatrix}
\bigg]  
\end{align*}
such that
\begin{enumerate}
\item[(a)] For all $N\in\mathbb{N}$, the law of $\overline{\bm{\Theta}}^N$ on $\chi$ is given by
\begin{align*}
\mathcal{L}\bigg[(W_t^k)_{k\in\mathbb Z}, \eta^N, \partial_t\eta^N, \mathbb{I}_{\mathcal{O}_{\eta^N}}\bu^N, \mathbb{I}_{\mathcal{O}_{\eta^N}}\nabx\bu^N,
\begin{pmatrix}
((\partial_t \eta^N, \mathscr{F}^{\eta^N}(\partial_t \eta^N)), (\partial_t\eta^N, \mathbb{I}_{\mathcal{O}_{\eta}} \bu^N))
\\
((0,\bu^N- \mathscr{F}^{\eta^N}(\partial_t\eta^N)), (\partial_t \eta^N, \mathbb{I}_{\mathcal{O}_{\eta}} \bu^N)^{•})
\end{pmatrix}
\bigg] ;
\end{align*}
\item[(b)] $\overline{\bm{\Theta}}^N$ converges $\overline{\mathbb{P}}$-almost surely to $\overline{\bm{\Theta}}$ in the topology of $\chi$, i.e.
\begin{align*}
\overline{W}^{k,N}_t\rightarrow \overline{W}^k_t  \quad\text{ in } \quad C(\overline{I})   \quad\overline{\mathbb{P}}\text{-a.s.},
\\
\overline{\eta}^N
\overset{*}{\rightharpoonup}
\overline{\eta} \quad
\text{ in } \quad L^\infty(I, W^{2,2}(\Gamma)) \quad  \overline{\mathbb{P}}\text{-a.s.},
\\
\partial_t\overline{\eta}^N
\overset{*}{\rightharpoonup}
\partial_t\overline{\eta} \quad
\text{ in } \quad L^\infty(I, L^{2}(\Gamma))   \quad\overline{\mathbb{P}}\text{-a.s.},
\\
\partial_t\overline{\eta}^N
{\rightharpoonup}
\partial_t\overline{\eta} \quad
\text{ in } \quad L^2(I, W^{1,2}(\Gamma))   \quad\overline{\mathbb{P}}\text{-a.s.},
\\
\overline{\eta}^N
\rightarrow
\overline{\eta} \quad
\text{ in } \quad L^\infty(I, L^\infty(\Gamma)) \quad  \overline{\mathbb{P}}\text{-a.s.},
\\
\mathbb{I}_{\mathcal{O}_{\overline{\eta}^N}}\overline{\bu}^N
\overset{*}{\rightharpoonup}
\mathbb{I}_{\mathcal{O}_{\overline{\eta}}} \overline{\bu} \quad
\text{ in } \quad L^\infty(I, L^{2}(\mathcal{O}\cup S_\alpha))   \quad\overline{\mathbb{P}}\text{-a.s.},
\\
\mathbb{I}_{\mathcal{O}_{\overline{\eta}^N}}\nabx\overline{\bu}^N
\rightharpoonup
\mathbb{I}_{\mathcal{O}_{\overline{\eta}}}\nabx\overline{\bu}
 \quad\text{ in } \quad L^2(I, L^2(\mathcal{O}\cup S_\alpha))   \quad\overline{\mathbb{P}}\text{-a.s.},
\end{align*}
as well as
\begin{align*}
\int_I\int_{\mathcal{O}\cup S_\alpha}
\mathbb{I}_{\mathcal{O}_{\overline{\eta}^N}}\overline{\bu}^N
\cdot
 \mathscr{F}^{\overline{\eta}^N}\partial_t\overline{\eta}^N\dx\dt
 +
 \int_I\int_\Gamma\vert\partial_t\overline{\eta}^N\vert^2\dy\dt
\\
\rightarrow
\int_I\int_{\mathcal{O}\cup S_\alpha}
\mathbb{I}_{\mathcal{O}_{\overline{\eta}}} \overline{\bu}
\cdot
 \mathscr{F}^{\overline{\eta}}\partial_t\overline{\eta}\dx\dt
 +
 \int_I\int_\Gamma\vert\partial_t\overline{\eta}\vert^2\dy\dt
\end{align*}
and
\begin{align*}
\int_I\int_{\mathcal{O}\cup S_\alpha}
\mathbb{I}_{\mathcal{O}_{\overline{\eta}^N}}\overline{\bu}^N
\cdot(\overline{\bu}^N-
 \mathscr{F}^{\overline{\eta}^N}\partial_t\overline{\eta}^N)\dx\dt
\\
\rightarrow
\int_I\int_{\mathcal{O}\cup S_\alpha}
\mathbb{I}_{\mathcal{O}_{\overline{\eta}}} \overline{\bu}
\cdot(\overline{\bu}-
 \mathscr{F}^{\overline{\eta}}\partial_t\overline{\eta})\dx\dt
\end{align*}
$\overline{\mathbb{P}}$-a.s.
\end{enumerate}
\end{proposition}
Now we introduce the filtration on the new probability space, 
which ensures the correct measurabilities of the new random variables.
For this, we let $(\overline{\mathfrak F}_t)_{t\geq0}$ and $(\overline{\mathfrak F}_t^{N})_{t\geq0}$ be the $\overline{\mathbb P}$-augmented
canonical filtration on the variables $\overline{\boldsymbol{\Theta}}$ and
$\overline{\boldsymbol{\Theta}}^N $, respectively, that is\footnote{Some of the variables are not continuous in time, for those one can define $\Sigma_t$ as the history of a random distribution, cf. \cite[Chapter 2.8]{BFHbook}.}
\begin{align*}
\overline{\mathfrak F}_t=\Sigma\big[\Sigma_t(\overline{W}^k_t)\cup \Sigma_t(\mathbb I_{\mathcal O_{\overline{\eta}}}\overline{\mathbf u}) \cup \Sigma_t(
\mathbb I_{\mathcal O_{\overline{\eta}}}\nabx\overline{\mathbf u}) \cup \Sigma_t( \overline\eta) \cup \Sigma_t(\partial_t\overline\eta ) 
\big]
\end{align*}
for $t\in I$ and similarly for $\overline{\mathfrak F}_t^N$.
By \cite[Theorem 2.9.1]{BFHbook}, the weak equation continuous to hold on the new probability space.

\section{Convergence in probability}
\label{sec:pcon}

\subsection{Identification of limit}
As in \cite[Section 3.3]{BrMeMo}, it  is straightforward to pass to the limit in all the drift  terms   
\begin{align*} 
\int_0^t
&\dd\bigg(\int_{\Gamma}\partial_t\overline{\eta}^N \,\iota_\eta\phi\dy
+
\int_{\mathcal{O}_{\overline{\eta}^N}}\overline{\bu}^N\cdot\mathcal{J}_\eta\bm{\phi}\dx
\bigg)
\\&-
\int_0^t\int_{\Gamma}(\partial_t\overline{\eta}^N\,\iota_\eta\partial_t\phi
- \naby^2 \overline{\eta}^N\naby^2\iota_\eta \phi  + (1+ \gamma)  \partial_t\overline{\eta}^N\naby^2\iota_\eta\phi
 +g\,\iota_\eta\phi
)\dy\ds
\\
&-\int_0^t\int_{\mathcal{O}_{\overline{\eta}^N}}\big(\overline{\bu}^N\cdot\partial_t(\mathcal{J}_\eta\bm{\phi}
)+
\overline{\bu}^N\otimes \overline{\bu}^N:\nabx \mathcal{J}_\eta\bm{\phi}
-
\nabx\overline{\bu}^N:\nabx\mathcal{J}_\eta\bm{\phi}
+\bff\cdot\mathcal{J}_\eta\bm{\phi}
\big)\dx\ds
\end{align*} 
to obtain the corresponding distributional solution of the limit system  $\overline{\mathbb{P}}$-a.s.   Thus, it remains to pass to the limit
the stochastic integral. By the Burkholder--Davis--Gundy inequality and the definition of $\sigma_k$ in \eqref{ansatz},  we deduce that   
\begin{align*}
\overline{\mathbb{E}}\bigg[ \sup_{t\in I}&\bigg\vert\int_0^t\sum_{k\in\mathbb{Z}} 
\int_{\Gamma} 
\sqrt{\gamma}
 \partial_t\overline{\eta}^N \sigma_k\,\naby \iota_\eta\phi
 \dy\dd W_s^k\bigg\vert \bigg]
 \lesssim 
\overline{\mathbb{E}}\bigg(\int_{I}  \sum_{k\in\mathbb{Z}}
\Big(\int_{\Gamma} \sqrt{\gamma}\partial_t\overline{\eta}^N \sigma_k\,\naby \iota_\eta\phi
 \dy\Big)^2\dt\bigg)^{1/2} 
  \\
 &\lesssim 
 \sqrt{\gamma} \Big(\sum_{N\leq|k|\leq N^2}\frac{1}{|k|}\Big)^{-1/2}
\overline{\mathbb{E}}\bigg(\int_{I}  \sum_{k\in\mathbb{Z}}
\Big(\int_{\Gamma} \partial_t\overline{\eta}^N \frac{ e^{\mathrm{i}ky }}{|k|}\bm{1}_{\{N\leq|k|\leq N^2\}}
 \naby \iota_\eta\phi
 \dy\Big)^2\dt\bigg)^{1/2}  
  \\
 &\lesssim 
\Big(\sum_{N\leq|k|\leq N^2}\frac{1}{|k|}\Big)^{-1/2}
\bigg( \sup_{N\leq|k|\leq N^2}\frac{1}{|k|}\bigg)^{1/2}  
 \overline{\mathbb{E}} \Big(\int_{I}  \Vert \partial_t\overline{\eta}^N  
  \naby \iota_\eta\phi
 \Vert_{L^2(\mathbb{T}^2)}^2\dt\Big)^{1/2}  
 \\
 &\lesssim  
\frac{(\log N)^{-1/2}}{ N }
\Vert \naby \iota_\eta\phi \Vert_{L^\infty(I\times \Gamma)} 
 \overline{\mathbb{E}}\Big( \int_I  \Vert \partial_t\overline{\eta}^N \Vert_{L^2(\Gamma)}^2 \dt  \Big)^{1/2} 
\lesssim\frac{1}{N \sqrt{\log N}  }
\end{align*} 
with a hidden constant depending only on $\gamma, T $ and $\Vert \naby \iota_\eta\phi \Vert_{L^\infty(I\times\Gamma)}$. In the last step, we have used H\"older inequality in time and the energy estimate \eqref{energyGal} whose right-hand side is uniform in $N$. The last estimate above, therefore, converges to zero as $N\rightarrow\infty$. Thus, we can conclude that the stochastic integral converges to zero when tested against a smooth functions. 
Hence we obtain
\begin{align*} 
\int_0^t
&\dd\bigg(\int_{\Gamma}\partial_t\overline{\eta} \,\iota_\eta\phi\dy
+
\int_{\mathcal{O}_{\overline{\eta}}}\overline{\bu}\cdot\mathcal{J}_\eta\bm{\phi}\dx
\bigg)
\\=&
\int_0^t\int_{\Gamma}(\partial_t\overline{\eta}^N\,\iota_\eta\partial_t\phi
- \naby^2 \overline{\eta}\naby^2\iota_\eta \phi  
+ (1+ \gamma)   \partial_t\overline{\eta}\naby^2\iota_\eta\phi
 +g\,\iota_\eta\phi)\dy\ds
\\
&+\int_0^t\int_{\mathcal{O}_{\overline{\eta}}}\big(\overline{\bu}\cdot\partial_t(\mathcal{J}_\eta\bm{\phi}
)+
\overline{\bu}\otimes \overline{\bu}:\nabx\mathcal{J}_\eta\bm{\phi}  
-
\nabx\overline{\bu}:\nabx\mathcal{J}_\eta\bm{\phi}+\bff\cdot\mathcal{J}_\eta\bm{\phi}\big)\dx\ds
\end{align*} 
for all  $(\phi, {\bfphi}) \in C^\infty(\overline{I}\times\Gamma) \times C^\infty_{\Div}(\overline{I}\times \R^2; \R^2)$ with $\bfphi\circ\bfvarphi =\phi {\bfn}$ on $I\times\Gamma$.
This yields
for all $\sigma\in I$
\begin{align}\label{eq.28.09B}
\begin{aligned}
\int_0^\sigma  \frac{\mathrm{d}}{\dt}\bigg(\int_\Gamma \partial_t \overline\eta \, \phi \dy
+
\int_{\mathcal{O}_{\overline{\eta}}}\overline{\bu} \cdot {\bfphi}\dx
\bigg)\dt 
&=
\int_0^\sigma \int_\Gamma \big(\partial_t \overline\eta\, \partial_t\phi
-
\partial_y^2\overline\eta\,\partial_y^2 \phi
 (1+ \gamma)  \partial_t\partial_y\overline\eta\,\partial_y \phi+
 g\, \phi \big)\dy\dt 
 \\&+
\int_0^\sigma  \int_{\mathcal{O}_{\overline{\eta}}}\big(  \overline{\bu}\cdot \partial_t  {\bfphi} + \overline{\bu} \otimes \overline{\bu}: \nabx {\bfphi} 
  \big) \dx\dt
\\&
-\int_0^\sigma  \int_{\mathcal{O}_{\overline{\eta}}}\big(   
\nabx \overline{\bu}:\nabx {\bfphi} -\bff\cdot{\bfphi} \big) \dx\dt
\end{aligned}
\end{align} 
for all  $\phi:I\times\Gamma\rightarrow\R$, $\bfphi:I\times\mathcal O_\eta\rightarrow\R^2$ sufficiently smooth\footnote{This will be made precise below.}
 with $\bfphi\circ\bfvarphi_\eta =\phi {\bfn}$ on $I\times\Gamma$.
We would now like to conclude that we have a weak solution to the deterministic problem in the sense of Definition \ref{def:weakSolution}. However, note carefully that we currently do not know that
$\partial_t\overline\eta$ is differentiable in space
and hence the integral $\gamma\int_\Gamma\partial_t\overline\eta\,\partial^2_y \phi\dy$ above is written in a very weak form (different to Definition \ref{def:weakSolution}).
Also, we must prove that the energy inequality is satisfied.  In a first step we aim to clarify why the class of test-functions in \eqref{eq.28.09B} is admissible.
By construction, the class of admissible test-functions includes
\begin{align*}
\mathrm{Span}\big(\{(\psi \iota_\eta w_i,\psi\mathcal J_{\eta}\bfw_i):\,i\in\mathbb N,\,\psi \in C_0^1([0,T))\}\big),
\end{align*}
where $(w_i)$ is a basis of $W^{2,2}(\Gamma)$ and $(\bfw_i)$ of $W^{1,2}_{\Div}(\mathcal{O})$ in such a way that $\bfw_i\circ\bfvarphi_\eta=w_i\bfn$ on $I \times\Gamma$. In \cite{BrMeMo}, the existence proof relies on an approximation, where the geometry $\mathcal{O}_\eta$ is replaced by a regularised geometry $\mathcal{O}_{\mathfrak r^\varepsilon\eta}$, where $\mathfrak r^\varepsilon$ is the standard mollification. By the properties of the Piola transform defined in \eqref{piolaTransform}, the set
\begin{align*}
\mathrm{Span}\big(\{(\psi \iota_{\mathfrak r^\varepsilon\eta} w_i,\psi\mathcal J_{\mathfrak r^\varepsilon\eta}\bfw_i):\,i\in\mathbb N,\,\psi \in C_0^1([0,T))\}\big)
\end{align*}
is dense in the set of couples
\begin{align}\label{eq:18.09C}
(\phi, {\bfphi}) \in W^{1,2}(I;L^2(\Gamma))\cap L^{2}(I;W^{2,2}(\Gamma)) \times L^2(I;W^{1,2}_{\Div}(\mathcal{O}_{\mathfrak r^\varepsilon\eta}))\cap W^{1,2}(I;L^2(\mathcal{O}_{\mathfrak r^\varepsilon\eta}))
\end{align}
which is further equivalent to the set
\begin{align*}
\mathrm{Span}\big(\{(\psi w_i,\psi\bfw_i):\,i\in\mathbb N,\,\psi \in C_0^1([0,T))\}\big)
\end{align*}
being dense in the set of couples
$$(\phi, {\bfphi}) \in W^{1,2}(I;L^2(\Gamma))\cap L^{2}(I;W^{2,2}(\Gamma)) \times L^2(I;W^{1,2}_{\Div}(\mathcal{O}))\cap W^{1,2}(I;L^2(\mathcal{O} ))$$ with $\phi(T,\cdot)=0$, ${\bfphi}(T,\cdot)=0$ and  $\bfphi\circ\bfvarphi=\phi {\bfn}$ on $I\times\Gamma$. The latter is well-known.
%

Consider $\phi\in W^{1,2}(I;L^2(\Gamma))\cap L^2(I;W^{2,2}(\Gamma))$ and $\varepsilon>0$.
We can insert $(\phi,\mathscr F_{\mathfrak r^\varepsilon \eta}\phi)$ in the class \eqref{eq:18.09C} using the estimates from Proposition \ref{prop:musc}. The latter also imply that
\begin{align*}
\mathscr F^{\mathfrak r^\varepsilon\eta }\phi\rightharpoonup \mathscr F^{\eta}\phi \quad\text{in}\quad W^{1,2}(I;L^2(\mathcal{O}\cup S_\alpha))\cap L^{2}(I;W^{1,2}_{\Div}(\mathcal{O}\cup S_\alpha)).
\end{align*}
Similarly, we have
\begin{align*}
\mathscr F^{\eta^N }\phi\rightharpoonup \mathscr F^{\eta}\phi \quad\text{in}\quad W^{1,2}(I;L^2(\mathcal{O}\cup S_\alpha))\cap L^{2}(I;W^{1,2}_{\Div}(\mathcal{O}\cup S_\alpha))
\end{align*}
using Proposition \ref{prop:skorokhod}. Thus the class of admissible test-functions includes tuple of the form
\begin{align*}
(\phi,\mathscr F^{\eta}(\phi)),\quad\phi\in W^{1,2}(I;L^2(\Gamma))\cap L^2(I;W^{2,2}(\Gamma)).
\end{align*}
Let us finally suppose that
\begin{align}\label{eq:19.08D}
\begin{aligned}
(\phi,\bfphi)\in W^{1,2}(I;L^2(\Gamma))&\cap L^{2}(I;W^{2,2}(\Gamma)) \\&\times W^{1,2}(I;L^2(\mathcal{O}_\eta))\cap L^{2}(I;W^{1,2}_{\Div}(\mathcal{O}_\eta))
\end{aligned}
\end{align}
as well as $\bfphi\circ\bfvarphi_\eta=\phi\bfn$ on $I\times\Gamma$.
First of all, we write
\begin{align*}
(\phi,\bfphi)=(\phi,\mathscr F_\eta\phi)+(0,\bfphi-\mathscr F_\eta\phi).
\end{align*}
We have already shown that $(\phi,\mathscr F_\eta\phi)$ is admissible. Note carefully that $\bfpsi:=(\bfphi-\mathscr F_\eta\phi)(t)\in W_{0,\Div}^{1,2}(\mathcal{O}_\eta)$ for a.a. $t\in I$.
Since the boundary of $\mathcal{O}_{\eta}$ is Lipschitz continuous by \eqref{energyGal},
we can approximate $(\bfpsi)(t)$ by $(\bfpsi(t))_\kappa$, where $ (\bfpsi(t))_\kappa \in C^\infty_{c,\Div}((\mathcal{O}_{\eta}))$, see \cite[Theorem III.4.1]{MR2808162}. The mapping $\bfpsi\mapsto (\bfpsi)_\kappa$ is linear and continuous in all customary spaces. Hence the differentiability in time is preserved. Finally, we denote by $\bfpsi_\kappa$ a temporal regularisation of $(\bfpsi)_\kappa$
which leads to a function which is also smooth in time. 
Note that $\bfpsi_\kappa$ is solenoidal and has compact support in space. Thus it is an admissible test-function. By standard properties of the regularisation we employed we can pass to the limit $\kappa\rightarrow0$ obtaining admissibility of $(0,\bfphi-\mathscr F_\eta\phi)$.

In conclusion, the class of admissible text-functions is given by \eqref{eq:19.08D}.

As a next step we aim to prove the energy balance.
As a first step towards this aim we improve the regularity of $\partial_t\overline{\eta}$.

\subsection{Higher regularity}
\label{sec:highreg}
We aim at improving the spatial regularity of $\overline{\eta}$ to
\begin{align}\label{eq:13.05B}
\overline{\eta}\in L^2(I;W^{5/2,2}(\Gamma)).
\end{align}
For this, we consider the operator $(-\partial_y^2)^{1/4}$ defined via Fourier modes
and the projection $\mathscr P_{\mathscr N}$ onto the first $\mathscr N$ Fourier modes.
Now, we consider
\begin{align*}
D^{\mathscr{K}}_{\mathscr N}f:=(-\partial_y^2)^{1/2}\mathscr P_{\mathscr N} f
-\mathscr{K}_{\overline\eta}((-\partial_y^2)^{1/2}\mathscr P_{\mathscr N}f),
\end{align*}
where $\mathscr K_{\overline\eta}$ is given in \eqref{eq:corr} and $(-\partial_y^2)^{l}f(x):=\sum_{k\in\mathbb Z}c_k(f)|k|^{2l}\exp(\mathfrak i kx)$
with the Fourier coefficients $(c_k(f))_{k\in\mathbb Z}$.
Finally, we choose
\begin{align*}
(\bm{\phi},\phi)=\big(\mathscr{F}^{\overline\eta}(D^{\mathscr{K}}_{\mathscr N}\overline{\eta})\,,D^{\mathscr{K}}_{\mathscr N}\overline{\eta}
\big)
\end{align*}
as test-function in the weak formulation \eqref{eq.28.09B} of the fluid-structure system. By making the fourth order term the subject, we obtain
\begin{align}
\int_0^t
\int_{\Gamma}&\nonumber
\partial_y^2 \mathscr{K}_{\overline\eta}((-\partial_y^2)^{1/2}\mathscr P_{\mathscr N}\overline{\eta}) \partial_y^2\overline{\eta} \dy\ds
-
\int_0^t
\int_{\Gamma}
\partial_y^2 ((-\partial_y^2)^{1/2}\mathscr P_{\mathscr N}\overline{\eta}) \partial_y^2\overline{\eta} \dy\ds
\\&\nonumber
- (1+ \gamma) 
\int_0^t
\int_{\Gamma}
\partial_y^2 \mathscr{K}_{\overline\eta}((-\partial_y^2)^{1/2}\mathscr P_{\mathscr N}\overline{\eta}) \partial_t\overline{\eta} \dy\ds
- (1+ \gamma) 
\int_0^t\int_{\Gamma}\partial_t\overline{\eta} \,
\partial_y^2((-\partial_y^2)^{1/2}\mathscr P_{\mathscr N}\overline{\eta})\dy\ds
\\
\nonumber=&
-
\bigg(
\int_{\mathcal{O}_{\overline{\eta}}}\overline{\bu}_0 \cdot\mathscr{F}^{\overline\eta}(D^{\mathscr{K}}_{\mathscr N}\overline{\eta}_0)\dx
+
\int_\Gamma   \overline\eta_\star \, (-\partial_y^2)^{1/2}\mathscr P_{\mathscr N}\overline{\eta}_0 \dy
\bigg) 
\\&\nonumber+
\int_{\mathcal{O}_{\overline{\eta}}}\overline{\bu} \cdot\mathscr{F}^{\overline\eta}(D^{\mathscr{K}}_{\mathscr N}\overline{\eta}(t))\dx
+
 \int_{\Gamma}   \partial_t\overline{\eta} (D^{\mathscr{K}}_{\mathscr N}\overline{\eta}(t))\dy
 \\&
 -
 \int_0^t \int_{\Gamma}   \partial_t\overline{\eta} \partial_t (D^{\mathscr{K}}_{\mathscr N}\overline{\eta})\dy\ds
-
\int_0^t
\int_{\mathcal{O}_{\overline{\eta}}}  \overline\bu  \cdot  \partial_t(\mathscr{F}^{\overline{\eta}}(D^{\mathscr{K}}_{\mathscr N}\overline{\eta})) 
\dx\ds\nonumber
\\&\nonumber
-
\int_0^t
\int_{\mathcal{O}_{\overline{\eta}}}  
  \overline{\bu}\otimes \overline{\bu} :\nabx \mathscr{F}^{\overline\eta}(D^{\mathscr{K}}_{\mathscr N}\overline{\eta})) 
 \dx\ds
 +
 \int_0^t
 \int_{\mathcal{O}_{\overline{\eta}}} 
 \nabx\overline{\bu}:\nabx\mathscr{F}^{\overline\eta}(D^{\mathscr{K}}_{\mathscr N}\overline{\eta})
 \dx\ds
 \\&-
 \int_0^t
  \int_{\mathcal{O}_{\overline{\eta}}} 
 \bff\cdot\mathscr{F}^{\overline\eta}(D^{\mathscr{K}}_{\mathscr N}\overline{\eta})
 \dx\ds
 -
 \int_0^t
  \int_{\Gamma} 
g\,D^{\mathscr{K}}_{\mathscr N}\overline{\eta}\dy\ds
\nonumber
 \\
=:&I_0+I_1+\ldots+I_{8}.\label{galerkinweak1A}
\end{align}
First of all, 
we note that
\begin{align*}
\int_{\Gamma}&
\big[\partial_y^2 ((-\partial_y^2)^{1/2}\mathscr P_{\mathscr N}\overline{\eta} )  \partial_y^2\overline{\eta} 
+ (1+ \gamma)  \partial_t\partial_y\overline{\eta}\, \partial_y((-\partial_y^2)^{1/2}\mathscr P_{\mathscr N}\overline{\eta}) \big]\dy 
\\
&=\int_{\Gamma}
\partial_y^2 ((-\partial_y^2)^{1/4}\mathscr P_{\mathscr N}\overline{\eta} ) (-\partial_y^2)^{1/4}\partial_y^2\mathscr P_{\mathscr N}\overline{\eta}  \dy 
\\ &+ (1+ \gamma) 
\int_\Gamma\partial_y((-\partial_y^2)^{1/4}\mathscr P_{\mathscr N}\overline{\eta})\partial_t\big(\partial_y(-\partial_y^2)^{1/4}\mathscr P_{\mathscr N}\overline{\eta} \big)\dy 
\\
&= 
 \int_{\Gamma}
 \vert(-\partial_y^2)^{5/4}\mathscr P_{\mathscr N}\overline\eta \vert^2 \dy 
+
\frac{(\gamma+1)}{2}
 \frac{\dd}{\dt}
 \int_{\Gamma}
 \vert(-\partial_y^2)^{3/4}  \mathscr P_{\mathscr N}\overline{\eta} \vert^2 \dy 
\end{align*}
and since the corrector $\mathscr{K}_{\overline\eta} $ is spatially independent,
\begin{align*}
\int_{\Gamma}
\partial_y^2 \mathscr{K}_{\overline\eta} ((-\partial_y^2)^{1/2}\mathscr P_{\mathscr N}\overline{\eta} ) \partial_y^2\overline{\eta}  \dy 
 =
0.
\end{align*}
Furthermore, we have
\begin{align*}
  I_1  
&\lesssim \sup_I\Vert  \overline{\bu}      \Vert_{L^{2}(\mathcal{O}_{\overline{\eta}})}
\sup_I \Vert
(-\partial_y^2)^{1/2}\mathscr P_{\mathscr N}\overline{\eta} 
\Vert_{L^{2}(\Gamma)}
\\
&\lesssim \sup_I\Vert \overline{\bu} \Vert_{L^{2}(\mathcal{O}_{\overline{\eta}})}^{2}
+
\sup_I \Vert
\overline{\eta}
\Vert_{W^{1,2}(\Gamma)}^{2}
\end{align*}
by Proposition \ref{prop:musc}, where the right-hand side is bounded by \eqref{energyGal} and Proposition  \ref{prop:skorokhod}.
Similarly,
\begin{align*}
  I_2 
&\lesssim 
\sup_I\Vert\partial_t \overline{\eta}  \Vert_{L^{2}(\Gamma)}
\sup_I \Vert
(-\partial_y^2)^{1/2}\mathscr P_{\mathscr N}\overline{\eta} 
\Vert_{L^{2}(\Gamma)}
\\
&\lesssim 
\sup_I\Vert \partial_t\overline{\eta}  \Vert_{L^{2}(\Gamma)}^{2}
+
\sup_I \Vert
\overline{\eta}
\Vert_{W^{1,2}(\Gamma)}^{2}
\end{align*}
which is again bounded. We write 
 \begin{align*}
-I_3=\int_I\int_\omega| \partial_t(-\partial_y^2)^{1/4}\overline\eta|^2\dy\dt+\int_I\int_\omega\partial_t\overline\eta\partial_t\mathscr K_{\overline\eta}((-\partial_y^2)^{1/2}\mathscr P_{\mathscr N}\overline{\eta})\dy\dt,
\end{align*}
the first inetgral is uniformly bounded by \eqref{energyGal}.
As for the second one, we apply \eqref{eq:corr} to obtain
\begin{align*}
\int_I\int_\omega&\partial_t \overline\eta\partial_t\mathscr K_{\overline\eta}((-\partial_y^2)^{1/2}\mathscr P_{\mathscr N}\overline{\eta})\dy\\&\lesssim \|\partial_t\overline\eta\|_{L^2(I;L^2(\omega))}\|\partial_t \overline\eta\|_{L^{2}(I;L^1(\omega))}\|\partial_t(-\partial_y^2)^{1/2}\mathscr P_{\mathscr N}\overline{\eta}\|_{L^2(I;L^1(\omega))}\\&\lesssim \|\partial_t\overline\eta\|_{L^2(I;L^2(\omega))}^2\|\partial_t\overline{\eta}\|_{L^2(I;W^{1,2}(\omega))}
,
\end{align*}
where the right-hand side is again bounded by \eqref{energyGal}.
We have by Proposition \ref{prop:musc} that
\begin{align*}
I_4 
&\lesssim
\int_I\|\partial_t\overline\eta(-\partial_y^2)^{1/2}\mathscr P_{\mathscr N}  \overline\eta\|_{L^{2}(\Gamma)}
\Vert \overline{\bu} \Vert_{L^{2}(\mathcal{O}_{\overline{\eta}})}
\dt
\\&
+
\int_I\|\partial_t(-\partial_y^2)^{1/2}\mathscr P_{\mathscr N}  \overline\eta\|_{L^{2}(\Gamma)}
\Vert \overline{\bu} \Vert_{L^{2}(\mathcal{O}_{\overline{\eta}})}
\dt
\\
&\lesssim
\int_I\|\partial_t\overline\eta\|^2_{L^2(\Gamma)} \|\partial_y\overline\eta\|_{L^{\infty}(\Gamma)}^2\dt+\int_I
\Vert \overline{\bu} \Vert_{L^{2}(\mathcal{O}_{\overline{\eta}})}^2
\dt
\\&
+
\int_I\|\partial_t  \overline\eta\|_{W^{1,2}(\Gamma)}
\dt,
\end{align*}
which can be controlled by
\eqref{energyGal} and Proposition  \ref{prop:skorokhod}.
Next, it follows from Proposition \ref{prop:musc} that
 (for $\delta>0$ arbitrary),
\begin{align*}
 I_5+I_6 
&\leq\delta
\int_I
\Vert\mathscr{F}^{\overline{\eta}}(D^{\mathscr{K}}_{\mathscr N}\overline{\eta} ) \Vert_{W^{1,2}(\mathcal{O}_{\overline{\eta}})}
^2\dt+c(\delta)\int_I\big(\Vert \overline{\bu} \Vert_{L^{4}(\mathcal{O}_{\overline\eta})}
^4+\Vert \overline{\bu} \Vert^2_{W^{1,2}(\mathcal{O}_{\overline{\eta}})}
\big)\dt
\\
&\lesssim
\delta\int_I 
 \Vert(-\partial_y^2)^{1/2}\mathscr P_{\mathscr N}  \overline\eta\Vert_{W^{1,2}(\Gamma)}^2\dt
+c(\delta)\\
&\lesssim
\delta\int_I 
 \Vert(-\partial_y^2)^{1/4}\mathscr P_{\mathscr N}  \overline\eta\Vert_{W^{2,2}(\Gamma)}^2\dt
+c(\delta),
\end{align*}
using also \eqref{energyGal} and Proposition  \ref{prop:skorokhod}. The forcing terms $I_7$ and $I_8$ are easily bounded thus if we combine everything and choose $\delta$ small enough, we obtain  
\begin{align}
\int_I\Vert \mathscr P_{\mathscr N}(-\partial_y^2)^{1/4}\overline{\eta} \Vert_{W^{2,2}(\Gamma )}^2\dt
\lesssim 1\label{GalerkinConvEta1NpS}
\end{align}
uniformly in $\mathscr N$
and thus \eqref{eq:13.05B}.
Combining this with \eqref{energyGal} and using a parabolic interpolation implies
\begin{align}\label{eq:13.05B'}
\overline{\eta}\in L^4(I;W^{9/4,2}(\Gamma))\hookrightarrow L^4(I;W^{2,4}(\Gamma)).
\end{align}

\subsection{The energy balance}
We rewrite \eqref{eq.28.09B} as an equation
for $(\partial_t\widetilde\eta,\widetilde\bv)$, where $\widetilde\eta:=\int_0^t\iota_{\overline{\eta}}\partial_t\overline\eta\,\dd s$ and $$\widetilde\bv:=\nabx\bfPsi_{\overline{\eta}}^\top(|\mathrm{det}(\nabx \bfPsi_{\overline{\eta}})|^{-1}\nabx\bfPsi_{\overline{\eta}})\widetilde\bu$$ with $\widetilde\bu:=\mathcal J_{\overline{\eta}}^{-1}\overline{\bu}$ which reads as 
\begin{equation}
\begin{aligned}
\label{weakeps'}
\frac{\dd}{\dt}  \bigg\langle\begin{pmatrix}\partial_t\widetilde\eta\\\widetilde\bv
\end{pmatrix}  
\cdot \begin{pmatrix} \phi
\\
\bm{\phi}
\end{pmatrix}
\bigg\rangle_{\mathbb H}
&=
\mathscr L  
\begin{pmatrix} 
\phi \\
\bm{\phi}
\end{pmatrix} 
\\
&:=\int_{\Gamma} \big(\partial_t\overline\eta \, \partial_t(\iota_{\overline\eta})\phi
{\color{blue}- (1+ \gamma)  }\partial_t\partial_y\overline{\eta}\partial_y \iota_{\overline{\eta}} \phi
 -
 \partial_y^2\overline{\eta}\partial_y^2 \iota_{\overline{\eta}} \phi
 +g\,\iota_{\overline{\eta}} \phi
\big)\dy
\\
&+
 \int_{\mathcal{O}_{\overline{\eta} (t)}}\Big(  \overline\bu\cdot \partial_t (\mathcal{J}_{\overline{\eta} (t)}) \bm{\phi} 
+\overline\bfu\otimes\overline\bu:\nabx (\mathcal{J}_{\overline{\eta}(t)}\bfphi )  \Big) \dx
\\
&-
 \int_{\mathcal{O}_{\overline{\eta} (t)}}\big(\nabx \overline{\bu}:\nabx (\mathcal{J}_{\overline{\eta} (t)}\bm{\phi})
 -
 \bff\cdot\mathcal{J}_{\overline{\eta} (t)}\bm{\phi}\big)  \dx,  
\end{aligned}
\end{equation}
where 
\begin{align*}
\mathbb H:=\big\{(\xi,\mathbf{w})\in W^{2,2}(\Gamma)\times W^{1,2}_{\divx}(\mathcal O):\,\mathbf{w} \circ\bm{\varphi}=\xi\bn\quad\text{on}\quad \Gamma \big\}.
\end{align*}
Note that $\mathbb H$ is a Hilbert space (as closed subset of the Hilbert space $W^{2,2}(\Gamma)\times W^{1,2}_{\divx}(\mathcal O)$) and $\mathscr L\in\mathbb H'$. This crucially hinges on \eqref{eq:13.05B'} which is just critical for
\begin{align*}
\mathcal J_\eta:L^\infty(I;L^2(\mathcal O))\cap L^2 (I;W^{1,2}(\mathcal O))\rightarrow L^\infty(I;L^2(\mathcal O_\eta))\cap L^2 (I;W^{1,2}(\mathcal O_\eta)).
\end{align*}
Hence we get from the fundamental theorem of calculus
applied to the mapping
\begin{align*}
t\mapsto \tfrac{1}{2}\int_{\mathcal O}\widetilde\bv\cdot
|\mathrm{det}(\nabx \bfPsi_{\overline{\eta}})|(\nabx\bfPsi_{\overline{\eta}}^\top
\nabx\bfPsi_{\overline{\eta}} )^{-1}
\widetilde\bv\dx
+
\tfrac{1}{2}\int_\Gamma\Big|\tfrac{\partial_t\widetilde\eta}{\iota_{\overline{\eta}}}\Big|^2\,\dy,
\end{align*}
noticing that $|\mathrm{det}(\nabx \bfPsi_{\overline{\eta}})|(\nabx\bfPsi_{\overline{\eta}}^\top
\nabx\bfPsi_{\overline{\eta}} )^{-1} 
\widetilde\bv=\widetilde\bu$,
\begin{align*}
\tfrac{1}{2}\int_{\mathcal O_{\overline{\eta}(t)}}|\overline\bu|^2\dx
&+\tfrac{1}{2}\int_\Gamma|\partial_t\overline\eta|^2\dy
\\&=
\tfrac{1}{2}\int_{\mathcal O}\widetilde\bv\cdot
|\mathrm{det}(\nabx \bfPsi_{\overline{\eta}})|(\nabx\bfPsi_{\overline{\eta}}^\top
\nabx\bfPsi_{\overline{\eta}} )^{-1}
\widetilde\bv\dx
+
\tfrac{1}{2}\int_\Gamma\Big|\tfrac{\partial_t\widetilde\eta}{\iota_{\overline{\eta}}}\Big|^2\dy
\\
&=
\tfrac{1}{2}\int_{\mathcal O}\widetilde\bv(0)\cdot\widetilde\bu(0)\dx+\tfrac{1}{2}\int_\Gamma|\eta_\star|^2\,\dy+\int_0^t\mathscr L\begin{pmatrix}\iota_{\overline{\eta}}^{-2}\partial_t\widetilde\eta\\\widetilde\bu
\end{pmatrix} \,\dd s
\\
&+\tfrac{1}{2}\int_{\mathcal O}\widetilde\bv\cdot\partial_t\big(|\mathrm{det}(\nabx \bfPsi_{\overline{\eta}})|(\nabx\bfPsi_{\overline{\eta}}^\top
\nabx\bfPsi_{\overline{\eta}} )^{-1}
\big)\widetilde\bv\dx-\int_\Gamma\tfrac{|\partial_t\widetilde\eta|^2}{\iota_{\overline{\eta}}^3}\partial_t\iota_{\overline{\eta}}\,\dy.
\end{align*}
Hence we obtain 
\begin{align*}
\tfrac{1}{2}\int_{\mathcal O_{\overline{\eta}(t)}}|\overline\bu|^2\dx+\tfrac{1}{2}\int_\Gamma|\partial_t\overline\eta|^2\dy
&=
\tfrac{1}{2}\int_{\mathcal O_{\overline{\eta}(0)}}|\bu_0|^2\dx+\tfrac{1}{2}\int_\Gamma|\eta_\star|^2\,\dy
\\
&-
\int_0^t
\int_{\Gamma} \partial_y^2 \iota_{ \overline\eta}\partial_t\widetilde\eta \,\partial_y^2\overline\eta\dy\ds
- (1+ \gamma) 
\int_0^t\int_{\Gamma}\partial_y\iota_{ \overline\eta}\partial_t\widetilde\eta \,\partial_y\overline\eta\dy\ds
\\&
-\int_0^t\int_{\mathcal O_{\overline\eta}}\nabx \bu:\nabx (\mathcal{J}_{\overline\eta (t)}\widetilde\bu)  \dx\ds
\\&
+
 \int_0^t\int_\Gamma g \,\iota_{ \overline\eta}\partial_t\widetilde\eta\dy \ds
 +
  \int_0^t\int_{\mathcal{O}_{\overline{\eta}}} \bff\cdot \mathcal{J}_{\overline\eta (t)}\widetilde\bu\dx \ds .
\end{align*}
Noticing that $\overline\bu=\mathcal{J}_{\overline\eta (t)}\widetilde\bu$ and $\partial_y^2 (\iota_{\overline\eta}\partial_t\eta)=\partial_t\partial_y^2\widetilde\eta$ and $\partial_y (\iota_{\overline\eta}\partial_t\eta)=\partial_t\partial_y\widetilde\eta$
 we finally obtain
\begin{align*}
\tfrac{1}{2}\int_{\mathcal{O}_{\overline{\eta}}} & |\overline\bu(t)|^2\dx
+
\int_0^t\int_{\mathcal{O}_{\overline{\eta}}}|\nabx\overline\bu|^2\dx\ds
+
\int_\Gamma\Big(\tfrac{1}{2}|\partial_t\overline\eta(t)|^2
+
\tfrac{1}{2}
|\partial_y^2\eta(t)|^2\Big)\dy
+ (1+ \gamma) \int_0^t\int_\Gamma|\partial_y\partial_t\overline\eta|^2\dy\ds
\\
&
= \tfrac{1}{2}\int_{\mathcal O_{\overline{\eta}(0)} }|{\bu}_0|^2\dx
+
\int_\Gamma\Big(\tfrac{1}{2}|\eta_\star|^2
+
\tfrac{1}{2}
|\partial_y^2\eta_0|^2\Big)\dy
+
 \int_0^t\int_\Gamma g \,\partial_t\overline\eta\dy \ds
 +
  \int_0^t\int_{\mathcal{O}_{\overline{\eta}}} \bff\cdot\overline{\bu}\dx \ds
\end{align*} 
for a.a. $t\in I$. Hence, by applying Young's inequality and Gr\"onwall lemma, we have  weak solution in the sense of Definition \ref{def:weakSolution} such that the weak-strong uniqueness result from \cite{BMSS35} applies.


\subsection{The Gy\"{o}ngy--Krylov argument}
\label{sec:cip}
In order to complete the proof of Theorem \ref{thm:main}, we make use of  \cite[Chapter 2, Theorem 2.10.3]{BFHbook} which is a generalisation of  the Gy\"{o}ngy--Krylov characterization of convergence in probability introduced in \cite{krylov} to quasi-Polish spaces. We consider two sequences $(N_n),(N_m)\subset \mathbb N$ diverging to infinity. Let
$${\boldsymbol{\vartheta}}^N:=\bigg[\mathbb I_{\mathcal O_{\eta^N}}\mathbf u^N,
\mathbb I_{\mathcal O_{\eta^N}} \nabx{\mathbf u}^N, \eta^N , 
\begin{pmatrix}
((\partial_t\eta^N,\mathscr F^{\eta^N}\partial_t\eta^N), (\partial_t\eta^N,\mathbb I_{\mathcal O_{\eta^N}}\mathbf u^N))
\\
((\partial_t \eta^N, \mathbf u
^N-\mathscr F^{\eta^N} \partial_t\eta^N), (\partial_t\eta^N,\mathbb I_{\mathcal O_{\eta^N}} \mathbf u^N))\end{pmatrix}\bigg],\quad N\in\mathbb N,$$
and set
$${\boldsymbol{\vartheta}}^n:={\boldsymbol{\vartheta}}^{N_n},\quad{\boldsymbol{\vartheta}}^m:={\boldsymbol{\vartheta}}^{N_m}.$$
We also set
$\bu^n:=\bu^{N_n}$, $\eta^n:=\eta^{N_n}$ and $\bu^m:=\bu^{N_m},\eta^m:=\eta^{N_m}$.
We consider the collection of joint laws of $({\boldsymbol{\vartheta}}^n,{\boldsymbol{\vartheta}}^m, (W^k_t)_{k\in\mathbb Z})$
on the extended path space
$$\chi_{\tt ext}:=\big(\chi_\bu\times\chi_{\nabx\bu}\times\chi_\eta\times\chi_{f,g}^2\big)^2\times \chi_W.$$ 
Similarly to Lemma \ref{lemma:fg'} we obtain tightness of 
$$\{\mathcal L[ {\boldsymbol{\vartheta}}^n,{\boldsymbol{\vartheta}}^m,(W^k_t)_{k\in\mathbb Z}];\,n,m\in\mathbb N\}$$
on $\chi_{\tt ext}$.

Let us take any subsequence $({\boldsymbol{\vartheta}}^{n_k},{\boldsymbol{\vartheta}}^{m_k}, W^k_t)$. By the Jakubowski-Skorokhod representation theorem we infer for a non-relabelled subsequence the existence of a probability space $(\bar{\Omega},\bar{\mathfrak F},\bar{\mathbb P})$ with a sequence of random variables  $(\hat{\boldsymbol{\vartheta}}^{n_k},\check{\boldsymbol{\vartheta}}^{m_k}, \overline{W}^k_t )$
converging almost surely in $\chi_{\tt ext}$ to a random variable $(\hat{\boldsymbol{\vartheta}},\check{\boldsymbol{\vartheta}}, (\overline{W}^k_t)_{k\in\mathbb Z})$.
Moreover,
$$\mathcal L[\hat{\boldsymbol{\vartheta}}^{n_k},\check{\boldsymbol{\vartheta}}^{m_k}, (\overline W_t^{k})_{k\in\mathbb Z}]=\mathcal L[{\boldsymbol{\vartheta}}^{n_k},{\boldsymbol{\vartheta}}^{m_k},(W_t^k)_{k\in\mathbb Z}]$$
on $\chi_{\tt ext}$ for all $k\in\mathbb N$.
%
Observe that, in particular, $\mathcal L[\hat{\boldsymbol{\vartheta}}^{n_k},\check{\boldsymbol{\vartheta}}^{m_k}, (\overline W_t^{k})_{k\in\mathbb Z}]$ converges weakly to the measure $\mathcal L[\hat{\boldsymbol{\vartheta}},\check{\boldsymbol{\vartheta}}, \overline (W_t^k)_{k\in\mathbb Z}]$.
As in the previous subsection we can
show that the limit objects
are (pathwise) weak solutions to \eqref{detNSF} in the sense of Definition \eqref{def:weakSolution}.
By the weak-strong uniqueness result from \cite{BMSS35} they must coincide with the strong solution (which exists by Theorem \ref{thm:Br}) and thus with each other. 
Therefore, the solutions $(\hat\eta,\hat\bu)$ and $(\check\eta,\check\bu)$ coincide $\bar{\mathbb P}$-a.s. and we have
\begin{align*}
\mathcal L&[\hat{\boldsymbol{\vartheta}},\check{\boldsymbol{\vartheta}}, (\overline W^k_t)_{k\in\mathbb Z}]\Big(({\boldsymbol{\vartheta}}_1,{\boldsymbol{\vartheta}}_2,(W^k_t)_{k\in\mathbb Z})\in\chi_{\tt ext}:\;(\eta_1,\bu_1)=(\eta_2,\bu_2)\Big)
=\bar{\mathbb P}((\hat\eta,\hat\bu)=(\check\eta,\check\bu))=1.
\end{align*}
Now, we have all in hand to apply the Gy\"ongy--Krylov theorem from \cite[Chapter 2, Theorem 2.10.3]{BFHbook}. 
It implies that the original sequence ${\boldsymbol{\vartheta}}^N$ defined on the initial probability space $(\Omega,\mathfrak F,\mathbb P)$ converges in probability in the topology of $\chi_\bu\times\chi_{\nabx\bu}\times\chi_\eta\times\chi_{f,g}^2$ to the random variable
$${\boldsymbol{\vartheta}}:=\bigg[\mathbb I_{\mathcal O_{\eta }} \mathbf u,
\mathbb I_{\mathcal O_{\eta }}\nabx{\mathbf u}, \eta,
\begin{pmatrix}
((\partial_t\eta,\mathscr F^{\eta}\partial_t\eta), (\partial_t\eta,\mathbb I_{\mathcal O_{\eta}}\mathbf u))
\\
((0, \mathbf u -\mathscr F^{\eta}\partial_t\eta), (\partial_t\eta,\mathbb I_{\mathcal O_{\eta }} \mathbf u))\end{pmatrix}\bigg].$$
In particular, we deduce that 
\begin{align}\label{eq:28.09}
\eta^N
\rightarrow
\eta \quad
\text{ in } \quad L^\infty(I, L^\infty(\Gamma)) \quad  \text{in probability}.
\end{align}

\subsection{No contact}\label{sec:no}
We observe that by using the Fundamental Theorem of Calculus, one-dimensional Agmon's inequality in space, and the Cauchy--Schwarz inequality in time, it follows that
\begin{align*}
\Vert\overline{\eta}(t)-\eta_0 \Vert_{L^\infty(\Gamma)} 
\leq \,\overline c
\int _{0}^{t}\Vert \partial _{t}\naby\overline{\eta}\Vert_{L^{2}(\Gamma)} \ds
\leq\,\overline c
\sqrt{t}\bigg(
\int _I\Vert \partial _{t}\naby\overline{\eta}\Vert_{L^{2}(\Gamma)}^2\dt\bigg)^{1/2} 
\end{align*}
for some positive constant $\overline c$.
Therefore, if we choose 
\begin{align*}
\gamma \geq \frac{\overline c^2T\big( \Vert \eta_\star \Vert_{L^2(\Gamma)}^2+  \Vert \naby^2\eta_0 \Vert_{L^2(\Gamma)}^2
+
\Vert \bu_0 \Vert_{L^2(\mathcal{O}_{\eta_0})}^2+\int_I\Vert g \Vert_{L^2(\Gamma)}^2\dt
+
\int_I\Vert \bff \Vert_{L^2(\mathcal{O}_{\overline{\eta}})}^2\dt
\big)}{(\tfrac{1}{2}(L-\alpha))^2},
\end{align*}
we obtain for $\Vert \eta_0\Vert_{L^\infty(\Gamma)} \leq \alpha <L$   that
\begin{equation}
\begin{aligned}
\Vert \overline{\eta} \Vert_{L^\infty(I\times\Gamma)} 
&\leq
\Vert \overline{\eta} -\eta_0\Vert_{L^\infty(I\times\Gamma)} 
+
\Vert \eta_0\Vert_{L^\infty(\Gamma)}  
\\
&\leq \overline c
\sqrt{\frac{T}{\gamma}}\sqrt{\gamma}
\Vert \partial_t\naby\overline{\eta} \Vert_{L^2(I;L^{2}(\Gamma))} 
+
\alpha
\\
&\leq \overline c
\sqrt{\frac{T}{\gamma}}  \big( \Vert \eta_\star \Vert_{L^2(\Gamma)}^2+  \Vert \naby^2\eta_0 \Vert_{L^2(\Gamma)}^2
+
\Vert \bu_0 \Vert_{L^2(\mathcal{O}_{\eta_0})}^2+1\big)^{1/2}    
+
\alpha
\\
&\leq \tfrac{L+\alpha}{2}.
\label{lAlp}
\end{aligned}
\end{equation} 
By using this estimate above, we note that by assuming $\Vert \overline{\eta}^N\Vert_{L^\infty(I\times\Gamma)}\geq L$, then
\begin{align*}
L\leq \Vert \overline{\eta}^N\Vert_{L^\infty(I\times\Gamma)}
\leq \Vert \overline{\eta}^N- \overline{\eta}\Vert_{L^\infty(I\times\Gamma)}+\tfrac{L+\alpha}{2}
\end{align*}
which implies that
\begin{align*}
\Vert \overline{\eta}^N- \overline{\eta}\Vert_{L^\infty(I\times\Gamma)}\geq \tfrac{L-\alpha}{2}.
\end{align*}
Consequently, the following containment hold
\begin{align*}
\big\{ \Vert \overline{\eta}^N \Vert_{L^\infty(I\times\Gamma)}\geq L\big\}
\subseteq
\big\{ \Vert \overline{\eta}^N- \overline{\eta}\Vert_{L^\infty(I\times\Gamma)}\geq \tfrac{L-\alpha}{2}\big\}
\end{align*}
when \eqref{lAlp} is true.
Thus, by monotonicity, we obtain
\begin{align*}
\overline{\mathbb{P}}\big(\Vert \overline{\eta}^N\Vert_{L^\infty(I\times\Gamma)}<L\big)
&=
1-\overline{\mathbb{P}}\big(\Vert \overline{\eta}^N\Vert_{L^\infty(I\times\Gamma)}\geq L\big)
\\&\geq
1-
\overline{\mathbb{P}}\big(\Vert \overline{\eta}^N-\overline\eta\Vert_{L^\infty(I\times\Gamma)}\geq \tfrac{L-\alpha}{2}\big).
\end{align*} 
Hence, for any $\kappa>0$,  we obtain from \eqref{eq:28.09} that for $N=N(\kappa)$ sufficiently large
\begin{align*}
\overline{\mathbb{P}}\big(\Vert \overline{\eta}^N\Vert_{L^\infty(I\times\Gamma)}<L\big)
\geq
1-\kappa.
\end{align*} 
This completes the proof of Theorem \ref{thm:main}.

%
%
 

\section*{Statements and Declarations} 
\subsection*{Funding}
This work has been partly supported by the Grants number 543675748, 525608987 and 563880468 by the German Research
Foundation (DFG).
\subsection*{Author Contribution}
The authors wrote and reviewed the manuscript.

\subsection*{Data Availability Statement}
Data sharing is not applicable to this article as no datasets were generated
or analyzed during the current study.
\subsection*{Competing Interests}
The authors have no competing interests to declare that are relevant to the content of this article.

\subsection*{AI Statement}  No AI has been used in the preparation of the manuscript.

%
%
%
%

%
%
%

\begin{thebibliography}{10}
\providecommand{\url}[1]{{#1}}
\providecommand{\urlprefix}{URL }
\expandafter\ifx\csname urlstyle\endcsname\relax
  \providecommand{\doi}[1]{DOI~\discretionary{}{}{}#1}\else
  \providecommand{\doi}{DOI~\discretionary{}{}{}\begingroup
  \urlstyle{rm}\Url}\fi

\bibitem{agresti1}Agresti, A.: Global smooth solutions by transport noise of 3D Navier-Stokes equations with small hyperviscosity.
To appear in Annals of Probability

\bibitem{agresti2} Agresti, A.: On the absence of blow-up in the 3D Navier-Stokes equations with transport noise. arxiv:2607.15140v4

\bibitem{fluid2014bodnar}
Bodn\'{a}r, T., Galdi, G.P., Ne\v{c}asov\'{a}, v. (eds.): Fluid-structure
  interaction and biomedical applications.
\newblock Advances in Mathematical Fluid Mechanics. Birkh\"{a}user/Springer,
  Basel (2014)

\bibitem{breit2022regularity}
Breit, D.: Regularity results in 2D fluid-structure interaction.
\newblock Math. Ann.   388, 1495--1538. (2024)
\newblock \doi{10.1007/s00208-022-02548-9}.


\bibitem{BFHbook}
Breit, D., Feireisl, E., Hofmanov\'{a}, M.: Stochastically forced compressible
  fluid flows, \emph{De Gruyter Series in Applied and Numerical Mathematics},
  vol.~3.
\newblock De Gruyter, Berlin (2018)

\bibitem{breit2021incompressible}
Breit, D., Mensah, P.R.: An incompressible polymer fluid interacting with a
  {K}oiter shell.
\newblock J. Nonlinear Sci. \textbf{31}(1), Paper No. 25, 56 (2021)

\bibitem{BrMeMo} Breit, D., Mensah  P. R.,  Moyo,  T. C.: Martingale solutions in stochastic fluid-structure interaction. J. Nonlinear Sci. \textbf{34}, Paper No 34. (2024)

\bibitem{BMSS35} Breit, D., Mensah,   P. R., Schwarzacher S., Su, P.: \emph{Ladyzhenskaya--Prodi--Serrin condition for fluid-structure interaction systems.} Ann. Sc. Norm. Super. Pisa Cl. Sci. (5), DOI:	10.2422/2036-2145.202407\textunderscore010

\bibitem{BMSS} Breit, D., Mensah,   P. R., Schwarzacher S., Su, P.: Well-posedness theorems in fluid-structure interaction: perfectly elastic shells. Preprint at arXiv:2605.14357

\bibitem{BrRo} Breit, D., Roy, A.: Compressible fluids and elastic plates in 2D: A conditional no-contact theorem. Phys. D \textbf{483}, 134967. (2025)


\bibitem{breitSchw2018compressible}
Breit, D., Schwarzacher, S.: Compressible fluids interacting with a
  linear-elastic shell.
\newblock Arch. Ration. Mech. Anal. \textbf{228}(2), 495--562 (2018)
%
%
%
%
%
%
%
%
%
%
%


\bibitem{Fl0}
Flandoli, F., Gubinelli, M., Priola, E.: Well-posedness of the transport
  equation by stochastic perturbation.
\newblock Invent. Math. \textbf{180}(1), 1--53 (2010)


\bibitem{flandoli2021scaling}
Flandoli, F., Galeati, L.,  Luo, D.: Scaling limit of stochastic 2D Euler equations with transport noises to the deterministic Navier?Stokes equations. Journal of Evolution Equations, 21(1), 567-600 (2021).

\bibitem{flandoli2020convergence}
Flandoli, F.,  Luo, D.: Convergence of transport noise to Ornstein--Uhlenbeck for 2D Euler equations under the enstrophy measure (2020)

\bibitem{Fl3}
Flandoli, F., Luo, D.: High mode transport noise improves vorticity blow-up
  control in 3{D} {N}avier-{S}tokes equations.
\newblock Probab. Theory Related Fields \textbf{180}(1-2), 309--363 (2021)

%


\bibitem{MR1873203}
Formaggia, L., Gerbeau, J.F., Nobile, F., Quarteroni, A.: On the coupling of
  3{D} and 1{D} {N}avier-{S}tokes equations for flow problems in compliant
  vessels.
\newblock Comput. Methods Appl. Mech. Engrg. \textbf{191}(6-7), 561--582
  (2001)
  

\bibitem{MR2808162}
Galdi, G.P.: An introduction to the mathematical theory of the
  {N}avier-{S}tokes equations, second edn.
\newblock Springer Monographs in Mathematics. Springer, New York (2011)

\bibitem{galeati2020convergence}
Galeati, L: On the convergence of stochastic transport equations to a deterministic parabolic one. Stochastics and Partial Differential Equations: Analysis and Computations, 8(4), 833-868 (2020).
%


\bibitem{MR3466847}
Grandmont, C., Hillairet, M.: Existence of global strong solutions to a
  beam-fluid interaction system.
\newblock Arch. Ration. Mech. Anal. \textbf{220}(3), 1283--1333 (2016)


\bibitem{krylov}
Gy\"{o}ngy, I., Krylov, N.: Existence of strong solutions for {I}t\^{o}'s
  stochastic equations via approximations.
\newblock Probab. Theory Related Fields \textbf{105}(2), 143--158 (1996)

%


\bibitem{Holm3}
Holm, D.D.: Variational principles for stochastic fluid dynamics.
\newblock Proc. A. \textbf{471}(2176), 20140,963, 19 (2015)

%


\bibitem{jakubow}
Jakubowski, A.: Short communication: The almost sure skorokhod representation
  for subsequences in nonmetric spaces.
\newblock Theory Probab. Appl. \textbf{42}(1), 167--174 (1998)

\bibitem{KuCa1}
Kuan, J., \v{C}ani\'{c}, S.: A stochastically perturbed fluid-structure
  interaction problem modeled by a stochastic viscous wave equation.
\newblock J. Differential Equations \textbf{310}, 45--98 (2022)

\bibitem{KuCa2}
Kuan, J., \v{C}ani\'{c}, S.: Well-posedness of solutions to stochastic fluid-structure interaction.
Journal of Mathematical Fluid Mechanics 26(1) 4 (2024)

\bibitem{KuTa1}
Kuan, J.,  Tawri, K. . Existence of weak martingale solutions to a stochastic fluid-structure interaction problem with a compressible viscous fluid. Journal of Differential Equations, 449, 113669 (2025).


\bibitem{lengeler2014weak}
Lengeler, D., R\r{u}\v{z}i\v{c}ka, M.: Weak solutions for an incompressible
  {N}ewtonian fluid interacting with a {K}oiter type shell.
\newblock Arch. Ration. Mech. Anal. \textbf{211}(1), 205--255 (2014)

\bibitem{Liu} Y Liu: Global strong solutions without contact for a viscous fluid interacting with a purely elastic beam. Preprint at arXiv:2610.22090

\bibitem{MeNg} Mensah, P.R., Ngougoue Ngougoue,  P. M.: Stochastically-constrained Koiter shell models. Stoch. Anal. Appl. (to appear)


\bibitem{muha2013existence}
Muha, B., Cani\'{c}, S.: Existence of a weak solution to a nonlinear
  fluid-structure interaction problem modeling the flow of an incompressible,
  viscous fluid in a cylinder with deformable walls.
\newblock Arch. Ration. Mech. Anal. \textbf{207}(3), 919--968 (2013)


\bibitem{MuSc}
Muha, B., Schwarzacher, S.: Existence and regularity for weak solutions for a
  fluid interacting with a non-linear shell in 3d.
\newblock Ann. Inst. H. Poincar\'e C Anal. Non Lin\'eaire \textbf{39},
  1369--1412 (2022)

\bibitem{on2}
Ondrej\'{a}t, M.: Stochastic nonlinear wave equations in local {S}obolev
  spaces.
\newblock Electron. J. Probab. \textbf{15}, no. 33, 1041--1091 (2010).

\bibitem{papadopoulos2009buckling}
Papadopoulos, V., Stefanou, G., Papadrakakis, M.: Buckling analysis of imperfect shells with stochastic non-Gaussian material and thickness properties. International Journal of Solids and Structures, 46(14-15), 2800-2 (2009)


\bibitem{reuter2025probabilistic}
Reuter, N.,  Kriegesmann, B.: Probabilistic analysis of cylindrical shells under continuously varying load combinations. Thin-walled structures, 215, 113319 (2025)

\bibitem{schafer2006stochastic}
Schafer, B. W., Graham-Brady, L.: Stochastic post-buckling of frames using Koiter's method. International Journal of Structural Stability and Dynamics, 6(03), 333-35 (2006)

\bibitem{SaaSch21}
Saari, O.,  Schwarzacher, S.: Construction of a right inverse for the divergence in non-cylindrical time dependent domains.
\newblock {\em Ann. PDE}, 9(1), 1--52 (2023)

\bibitem{ScSu}
Schwarzacher, S., Su, P.: Regularity for an elastic beam interacting with 2d
  navier-stokes equations.
\newblock Comm. Math. Phys. (2025)

\bibitem{Ta1}
Tawri, K.: A 2D stochastic nonlinearly coupled fluid-structure interaction problem in compliant arteries with unrestricted structural displacement. Journal of Differential Equations, 431, 113243  (2025).


\bibitem{KuCa3}
Tawri, K., \v{C}ani\'{c}, S.: Existence of martingale solutions to a
  nonlinearly coupled stochastic fluid-structure interaction problem.
\newblock Comm. PDE, 50(3), 353–406.  (2025)

\end{thebibliography}

\end{document}